\documentclass[a4paper,11pt]{article}
\usepackage{amsmath,amsthm,amssymb,amsfonts,array,fullpage,graphicx,mathrsfs,mathtools}

\allowdisplaybreaks

\newtheorem{thm}{Theorem}[section]

\newtheorem{dfn}[thm]{Definition}

\newtheorem{lem}[thm]{Lemma}
\newtheorem{propn}[thm]{Proposition}
\newtheorem{prop}[thm]{Proposition}
\newtheorem{rem}[thm]{Remark}

\newcommand{\Pb}{\mathbf{P}}
\newcommand{\mz}{\mathbb{Z}^{d}}
\newcommand{\hpl}[2]{\mathbb{H}^{(#1)}_{#2}}

\begin{document}

\title{Annealed transition density of simple random walk\\on a high-dimensional loop-erased random walk}
\author{D.\ A.\ Croydon\thanks{Research Institute for Mathematical Sciences, Kyoto University, Japan, croydon@kurims.kyoto-u.ac.jp}, D.\ Shiraishi\thanks{Graduate School of Informatics, Kyoto University, Kyoto, Japan, shiraishi@acs.i.kyoto-u.ac.jp}, S.\ Watanabe\thanks{Graduate School of Informatics, Kyoto University, Kyoto, Japan, watanabe@acs.i.kyoto-u.ac.jp}}
\maketitle

\begin{abstract}
\noindent
We derive sub-Gaussian bounds for the annealed transition density of the simple random walk on a high-dimensional loop-erased random walk. The walk dimension that appears in these is the exponent governing the space-time scaling of the process with respect to the extrinsic Euclidean distance, which contrasts with the exponent given by the intrinsic graph distance that appears in the corresponding quenched bounds. We prove our main result using novel pointwise Gaussian estimates on the distribution of the high-dimensional loop-erased random walk.\\
\textbf{AMS 2020 MSC:} 60K37 (primary), 05C81, 60J35, 60K35, 82B41.
\end{abstract}

\section{Introduction}

The main result of this article provides annealed heat kernel estimates for the random walk on the trace of a loop-erased random walk in high dimensions. Whilst this is something of a toy model, the statement reveals an interesting difference between the quenched (typical) heat kernel estimates, which are of Gaussian form with respect to the intrinsic graph metric, and the annealed (averaged) heat kernel estimates, which are of sub-Gaussian form with respect to the extrinsic Euclidean metric. (We will present more detailed background concerning this terminology below.) Establishing such a discrepancy rigourously was motivated by a conjecture made in \cite[Remark 1.5]{BCK2} for the two-dimensional uniform spanning tree, and naturally leads one to consider to what extent the behaviour is typical for random walks on random graphs embedded into an underlying space.

Before describing further our heat kernel estimates, we would like to present the key input into their proof, which is a time-averaged Gaussian bound on the distribution of the loop-erased random walk. In particular, throughout the course of the article, we let $(L_n)_{n\geq 0}$ be the loop-erasure of the discrete-time simple random walk $(S_n)_{n\geq 0}$ on $\mathbb{Z}^d$, where $d\geq 5$, started from the origin. (See Section \ref{sec2} for a precise definition of this process, which was originally introduced by Lawler in \cite{Lawduke}.) In the dimensions we are considering, it is known that the loop-erased random walk (LERW) behaves similarly to the original simple random walk (SRW), in that both have a Brownian motion scaling limit; for the SRW, this follows from the classic invariance principle of Donsker \cite{Donsker}, whereas for the LERW, it was proved in \cite{Lawduke}. Taking into account the diffusive scaling, such results readily imply convergence of the distributions of $n^{-1/2}S_n$ and $n^{-1/2}L_n$ as $n\rightarrow\infty$. For our heat kernel estimates, we will require understanding of the distribution of LERW on a finer, pointwise scale. Now, one can check that SRW satisfies pointwise Gaussian bounds of the form
\begin{equation}\label{srwbound0}
c  n^{-d/2}e^{-\frac{|x|^2}{cn}}\mathbf{1}_{\{n\geq \|x\|_1\}}\leq\frac{\mathbf{P}(S_n=x)+\mathbf{P}(S_{n+1}=x)}{2}\leq c^{-1} n^{-d/2}e^{-\frac{c|x|^2}{n}}, \qquad \forall x\in \mathbb{Z}^d,\: n\geq 1,
\end{equation}
where $c$ is a constant and we write $\|x\|_1$ for the $\ell_1$-norm of $x$, see \cite[Theorem 6.28]{Barbook}, for example. (The averaging over two time steps is necessary for parity reasons.) Of course, one can not expect the same bounds for a LERW. Indeed, the `on-diagonal' part of the distribution $\mathbf{P}(L_n=0)$ is equal to zero for $n\geq 1$. What we are able to establish is the following, which demonstrates that if one averages over longer time intervals, then one still sees Gaussian behaviour.

\begin{thm}\label{thm:main1}
The loop-erased random walk $(L_n)_{n\geq 0}$ on $\mathbb{Z}^d$, $d\geq 5$, started from the origin satisfies the following bounds: for all $x\in \mathbb{Z}^d\backslash\{0\}$, $n\geq 1$,
\[\frac{1}{n}\sum_{m=n}^{2n-1}\mathbf{P}(L_m=x)\leq c_1 n^{-d/2}e^{-\frac{c_2|x|^2}{n}},\]
and for all $x\in \mathbb{Z}^d\backslash\{0\}$, $n\geq |x|$,
\[\frac{1}{n}\sum_{m=\lceil c_3n\rceil}^{\lfloor c_4n\rfloor}\mathbf{P}(L_m=x)\geq c_5  n^{-d/2}e^{-\frac{c_6|x|^2}{n}},\]
where $c_1,\dots,c_6$ are constants.
\end{thm}

\begin{rem}
Another random path known to scale to Brownian motion in dimensions $\geq 5$ is the self-avoiding walk (SAW) \cite{HarSl, HarSl2}. A local central limit theorem for a `spread out' version of the SAW was established in \cite{HS}. For a precise statement, one should refer to \cite[Theorem 1.3]{HS}, but roughly it was shown that if $c_n(x)$ is the number of $n$-step self-avoiding walk paths from 0 to $x$ and $c_n:=\sum_{x\in\mathbb{Z}^d} c_n(x)$, then, for appropriate constants $c_1$ and $c_2$, as $n\rightarrow \infty$,
\[\frac{1}{(2r_n+1)^d}\sum_{y\in B_\infty(\sqrt{n} x, r_n)}\frac{c_n(y)}{c_n}\sim c_1n^{-d/2}e^{-\frac{c_2|x|^2}{2}},\]
where $B_\infty(x,r)$ is an $\ell_\infty$-ball and $r_n$ is a divergent sequence satisfying $r_n=o(n^{1/2})$. (We note that the term $z_c^n$ in \cite{HS} can be replaced by $1/c_n$ by applying \cite[Theorem 1.1]{HS2}.) In particular, the left-hand side here can be thought of as the probability a self-avoiding walk of length $n$ is close to $x$, with the ball over which the spatial averaging is conducted being microscopic compared with the diffusive scaling of the functional scaling limit. For the LERW of this article, one might ask if it was similarly possible to reduce the intervals over which time averaging is conducted in Theorem \ref{thm:main1} to ones of the form $\{n,n+1,\dots,n+t_n\}$ where $t_n=o(n)$. And, for both the LERW and the SAW, one could ask for what ranges of the variables a pointwise space-time bound as at \eqref{srwbound0} holds. See \cite{ABR} for a result in this direction concerning the weakly self-avoiding walk.
\end{rem}

We now introduce our model of random walk in a random environment. To this end, given a realisation of $(L_n)_{n\geq 0}$, we define a graph $\mathcal{G}$ to have vertex set
\[V(\mathcal{G}):=\left\{L_n:\:n\geq 0\right\},\]
and edge set
\[E(\mathcal{G}):=\left\{\{L_n,L_{n+1}\}:\:n\geq 0\right\}.\]
We then let $(X^\mathcal{G}_t)_{t\geq 0}$ be the continuous-time random walk on $\mathcal{G}$ that has unit mean exponential holding times at each site and jumps from its current location to a neighbour with equal probability being placed on each of the possibilities. Moreover, we will always suppose that $X_0^\mathcal{G}=L_0=0$. Clearly, since $\mathcal{G}$ equipped with its shortest-path graph metric $d_\mathcal{G}$ is isometric to $\mathbb{Z}_+$, it is straightforward to deduce that, for each realisation of $L$,
\begin{equation}\label{bms}
\left(n^{-1/2}d_\mathcal{G}\left(0,X_{nt}^\mathcal{G}\right)\right)_{t\geq 0}\rightarrow \left(|B_{t}|\right)_{t\geq 0},
\end{equation}
in distribution, where $(B_t)_{t\geq 0}$ is a one-dimensional Brownian motion started from 0, and we can also give one-dimensional Gaussian bounds for the transition probabilities of $X^\mathcal{G}$ on $\mathcal{G}$ (see Lemma \ref{qtlem} below). In this article, though, we are interested in the annealed situation, when we average out over realisations of $\mathcal{G}$. Precisely, we define the annealed law of $X^\mathcal{G}$ to be the probability measure on the Skorohod space $D(\mathbb{R}_+,\mathbb{R}^d)$ given by
\[\mathbb{P}\left(X^\mathcal{G}\in \cdot\right)=\int {P}^\mathcal{G}\left(X^\mathcal{G}\in \cdot\right)\mathbf{P}(d\mathcal{G}),\]
where $\mathbf{P}$ is the probability measure on the underlying probability space upon which $L$ is built, and $P^\mathcal{G}$ is the law of $X^\mathcal{G}$ on the particular realisation of $\mathcal{G}$ (i.e.\ the quenched law of $\mathcal{G}$). We are able to prove the following. Note that we use the notation $x\vee y:=\max\{x,y\}$ and $x\wedge y:=\min\{x,y\}$.

\begin{thm}\label{mainres}
For any $\varepsilon>0$, there exist constants $c_1,c_2,c_3,c_4\in(0,\infty)$ such that, for every $x\in\mathbb{Z}^d$ and $t\geq \varepsilon |x|$,
\[\mathbb{P}\left(X_t^\mathcal{G}=x\right)\leq c_1\left(1\wedge|x|^{2-d}\right)\left(1\wedge t^{-1/2}\right)\exp\left(-c_2\left(\frac{|x|^4}{1\vee t}\right)^{1/3}\right)\]
and also
\[\mathbb{P}\left(X_t^\mathcal{G}=x\right)\geq c_3\left(1\wedge|x|^{2-d}\right)\left(1\wedge t^{-1/2}\right)\exp\left(-c_4\left(\frac{|x|^4}{1\vee t}\right)^{1/3}\right).\]
\end{thm}
\bigskip

\begin{rem}
Essentially similar arguments would apply if $X^\mathcal{G}$ was replaced by the discrete-time simple random walk on $\mathcal{G}$. In that case, one would need to assume $t\geq \|x\|_1$ for the lower bound, since the probability being estimated would be zero for $t<\|x\|_1$. In the continuous setting, the behaviour of $\mathbb{P}(X_t^\mathcal{G}=x)$ takes on a different form for $t<\varepsilon |x|$, as the probability will be determined by certain rare events (cf.\ the discussion of Poisson bounds in \cite[Section 5.1]{Barbook}; see also Lemma \ref{qtlem} below).
\end{rem}

To put this result into context, it helps to briefly recall what kind of behaviour has been observed for anomalous random walks and diffusions in other settings. In particular, for many random walks or diffusions on fractal-like sets (either deterministic or random), it has been shown that the associated transition density $(p_t(x,y))$ satisfies, within appropriate ranges of the variables, upper and lower bounds of the form
\begin{equation}\label{sgb}
c_1t^{-d_s/2}\exp\left(-c_2\left(\frac{d(x,y)^{d_w}}{t}\right)^{\frac{1}{d_w-1}}\right),
\end{equation}
where $d(x,y)$ is some metric on the space in question. (See \cite{Barbook, Kumbook} for overviews of work in this area.) The exponent $d_s$ is typically called the spectral dimension, since it is related to growth rate of the spectrum of the generator of the stochastic process. The exponent $d_w$, which is usually called the walk dimension (with respect to the metric $d$), gives the space-time scaling.

Now, in our setting, we can clearly write
\[\mathbb{P}\left(X_t^\mathcal{G}=x\right)=\mathbb{P}\left(X_t^\mathcal{G}=x\:\vline\:x\in\mathcal{G}\right)\mathbf{P}\left(x\in\mathcal{G}\right),\]
and, moreover, using simple facts about the intersection properties of SRW in high dimensions, one can check that $\mathbf{P}(x\in\mathcal{G})\asymp 1\wedge|x|^{2-d}$ (where we use the notation $\asymp$ to mean that the left-hand side is bounded above and below by constant multiples of the right-hand side). Hence, Theorem \ref{mainres} gives that $\mathbb{P}(X_t^\mathcal{G}=x\:\vline\:x\in\mathcal{G})$ satisfies upper and lower bounds of exactly the sub-Gaussian form at \eqref{sgb}, with $d_s=1$, $d_w=4$ and $d$ being the Euclidean metric. We can understand that $d_s=1$ results from one-dimensional nature of the graph $\mathcal{G}$ with respect to its intrinsic metric $d_\mathcal{G}$. Moreover, the exponent $d_w=4$ gives the space-time scaling of the process $X^\mathcal{G}$ with respect to the Euclidean metric. Indeed, it is not difficult to combine \eqref{bms} and the Brownian motion scaling limit of $L$ (from \cite{Lawduke}) to check that, up to a deterministic constant factor in the time scaling,
\[\left(n^{-1/4}X_{nt}^\mathcal{G}\right)_{t\geq 0}\rightarrow \left(W_{|B_{t}|}\right)_{t\geq 0},\]
where $(W_t)_{t\geq 0}$ is a $d$-dimensional Brownian motion, independent of the one-dimensional Brownian motion $B$, with $W$ and $B$ each started from the origin. (The limit process is a version of the $d$-dimensional iterated Brownian motion studied in \cite{DeB}, for example. Cf.\ the result in \cite{rwrrw} for random walk on the range of a random walk.) We note that the exponent $d_s=1$ matches the quenched spectral dimension, whereas the $d_w=4$ is the multiple of the `2' of the quenched bound, which is the walk dimension of $X^\mathcal{G}$ with respect to the intrinsic metric $d_\mathcal{G}$, and the `2' that gives the space-time scaling of $L$. We highlight that the annealed bound is not obtained by simply replacing $d_\mathcal{G}(0,x)$ by $|x|^2$ in the quenched bound, though, as doing that does not result in an expression of the form at \eqref{sgb}.

As remarked at the start of the introduction, it was conjectured in \cite{BCK2} that a similar combination of the various exponents will appear in sub-Gaussian annealed heat kernel bounds for the random walk on the two-dimensional uniform spanning tree. In that case, the spectral dimension of the quenched and annealed bounds is known to be $16/13$, the intrinsic walk dimension is $13/5$ and the exponent governing the embedding is  given by the growth exponent of the two-dimensional LERW, i.e.\ $5/4$, giving an extrinsic walk dimension of $13/4$. We are able to check the corresponding result for our simpler model using the simple observation that
\begin{equation}\label{decomp}
\mathbb{P}\left(X_t^\mathcal{G}=x\right)=\sum_{m\geq 0}P^\mathcal{G}(X^\mathcal{G}=L_m)\mathbf{P}(L_m=x),
\end{equation}
and then combining the estimate on the distribution of the LERW from Theorem \ref{thm:main1} with the deterministic Gaussian bounds on $P^\mathcal{G}(X^\mathcal{G}=L_m)$ of Lemma \ref{qtlem} below. It is clear that a similar argument would apply for other random walks on random paths, including ones with different exponents, whenever one can suitably estimate the term corresponding to $\mathbf{P}(L_m=x)$. LERW in dimensions $2$, $3$ and $4$ would represent worthwhile examples to consider here. For more general random graphs, one could give a decomposition that is similar to \eqref{decomp}, but to proceed from that point one would need both good quenched heat kernel estimates with respect to some distance in place of the Gaussian bounds on $P^\mathcal{G}(X^\mathcal{G}=L_m)$ and good estimates on how the distance that appears in the quenched bounds behaves. Pursuing this program would be of interest for various random graphs that do not exhibit homogenisation. Apart from the uniform spanning tree/forest (in two and other dimensions, see \cite{ACHS,BCK,BCK2} for some relevant background on the associated random walk), one might consider the situation for random walk on: the trace of a simple random walk on $\mathbb{Z}^d$, the scaling limit of which was derived for $d\geq 4$ in \cite{rwrrw,rwrrw4}; the high-dimensional branching random walk, for which a scaling limit was proved in \cite{BCF2, BCF1}; and on high-dimensional critical percolation clusters, for which a number of related random walk estimates have previously been established \cite{HHH,KN}, for example.

The remainder of the article is organised as follows. In Section \ref{sec2}, we define the LERW $(L_n)_{n\geq 0}$ properly, and present some preliminary estimates that will be useful for understanding this process. After this,  in Section \ref{sec3}, we study the LERW in more detail, proving Theorem \ref{thm:main1}. Finally, in Section \ref{sec4}, we derive our heat kernel estimates for $X^\mathcal{G}$, proving Theorem \ref{mainres}.

\section{Preliminaries}\label{sec2}

In this section, we will start by presenting some notation, and then precisely define the loop-erased random walk that is the main object of interest in this article. We also give some estimates on hitting probabilities for annuli by a simple random walk that will be useful in Section \ref{sec3}.

\subsection{Subsets}

Throughout this paper, we always assume that $d \ge 5$. If $A \subseteq \mathbb{Z}^{d}$, we set
\begin{align*}
\partial A      &= \left\{ x \in \mathbb{Z}^{d} \setminus A  \::\: \exists y \in A \text{ such that } |x- y| = 1 \right\}, \\
\partial_{i} A  &= \left\{ x \in  A  \::\:  \exists y \in \mathbb{Z}^{d} \setminus A \text{ such that } |x- y| = 1 \right\}
\end{align*}
for the outer and inner boundary of $A$, respectively. We write $\overline{A} = A \cup \partial A$ for the discrete closure of $A$. We often use $A^{c}$ to denote $\mathbb{Z}^{d} \setminus A$.

For $x \in \mathbb{R}^{d}$ and $r > 0$, we let
\begin{align*}
 B (x, r) = \left\{ y \in \mathbb{Z}^{d} \::\: |x - y| \le r \right\},\qquad  B_{\infty} (x, r) = \left\{ y \in \mathbb{Z}^{d} \::\: \| x - y \|_{\infty} \le r \right\}
\end{align*}
be the discrete ball of radius $r$ centered at $x$ with respect to the Euclidean distance $| \cdot |$ and $\ell_{\infty}$-norm $\| \cdot \|_{\infty}$ in $\mathbb{R}^{d}$, respectively. We often call $B_{\infty} (x, r)$ the box of side length $2r$ centered at $x$.

If $A$ and $B$ are two subsets of $\mathbb{Z}^{d}$, we define the distance between them by setting
\[\text{dist} (A, B) = \inf_{x \in A, y \in B} |x-y|.\]

\subsection{Paths and loop-erasure}

Let $m \ge 0$ be an integer. If $\lambda = [\lambda_{0}, \lambda_{1}, \dots , \lambda_{m} ] \subset \mathbb{Z}^{d}$ is a sequence of points in $\mathbb{Z}^{d}$ satisfying $|\lambda_{i} - \lambda_{i-1} | = 1$ for $1 \le i \le m$, we call $\lambda$ a path of length $m$. The length of $\lambda$ is denoted by $\text{len} (\lambda)$. We write $\lambda [i, j] = [\lambda_{i}, \lambda_{i+1}, \dots , \lambda_{j}]$ for $0\leq i \le j\leq \text{len}(\lambda)$. If $\lambda$ satisfies $\lambda_{i} \neq \lambda_{j}$ for all $i \neq j$, then $\lambda$ is called a simple path. If $\lambda = [\lambda_{0}, \lambda_{1}, \dots ] \subset \mathbb{Z}^{d}$ is an infinite sequence of points in $\mathbb{Z}^{d}$ such that $[\lambda_{0}, \lambda_{1}, \dots , \lambda_{m}]$ is a path for each $m \ge 0$, we call $\lambda$ an infinite path. We write $\lambda [i, \infty) = [\lambda_{i}, \lambda_{i+1}, \dots  ]$ for $i \ge 0$.

If $\lambda$ is a path in $\mathbb{Z}^{d}$ (either finite or infinite) and $A \subseteq \mathbb{Z}^{d}$, we define
\begin{equation}\label{hitting}
\tau^{\lambda}_{A} = \inf\left\{ j \ge 0 \::\: \lambda_{j} \in A \right\}
\end{equation}
to be the first time that $\lambda$ hits $A$, with the convention that $\inf \emptyset = + \infty$. If $A=\{x\}$, then we set $\tau^\lambda_x=\tau^\lambda_{\{x\}}$.

Given a path $\lambda = [\lambda_{0}, \lambda_{1}, \dots , \lambda_{m} ] \subset \mathbb{Z}^{d}$ with $\text{len} (\lambda ) =m$, we define its (chronological) loop-erasure  $\text{LE} (\lambda )$ as follows. Let $\sigma_{0} = \max \{ k \::\: \lambda_{k} = \lambda_{0}\}$  and also, for $i\geq 1$,
\begin{equation}\label{sigmadef}
\sigma_{i} = \max \left\{ k \::\: \lambda_{k} = \lambda_{\sigma_{i-1} +  1 } \right\}.
\end{equation}
We note that these quantities are well-defined up to the index $j = \min \{ i \::\: \lambda_{\sigma_{i} } = \lambda_{m} \}$, and we use them to define the loop-erasure of $\lambda$ by setting
\[\text{LE} (\lambda) = \left[\lambda_{\sigma_{0}}, \lambda_{\sigma_{1} }, \dots , \lambda_{\sigma_{j} } \right].\]
It follows by construction that $\text{LE} (\lambda)$ is a simple path satisfying $\text{LE} (\lambda) \subseteq \lambda$, $\text{LE} (\lambda)_{0} = \lambda_{0}$ and $\text{LE} (\lambda)_{j} = \lambda_{m} $. If $\lambda = [\lambda_{0}, \lambda_{1}, \dots ] \subseteq \mathbb{Z}^{d}$ is an infinite path such that $\{ k\::\:\lambda_{k} = \lambda_{i} \}$ is finite for each $i \ge 0$, then its loop-erasure $\text{LE} (\lambda)$ can be defined similarly.

\subsection{Random walk and loop-erased random walk}

As in the introduction, let $S=(S_n)_{n\geq0}$ be the discrete-time simple random walk (SRW) on $\mathbb{Z}^{d}$. For $x \in \mathbb{Z}^{d}$, the law of $S$ started from $S_{0} = x$ will be denoted by $\Pb^{x}$. The expectation with respect to $\Pb^{x}$ will be denoted by $\mathbf{E}^{x}$. We will also define $\Pb = \Pb^{0}$ and $\mathbf{E} = \mathbf{E}^{0}$.

When $d \ge 5$, it is well known that $S$ is transient. Hence for every $x \in \mathbb{Z}^{d}$ and $\Pb^{x}$-almost-every realisation of $S$, it is possible to define the loop-erasure of $S[0, \infty )$, and we will denote this infinite path by
\[L = \text{LE} \left( S [0, \infty ) \right).\]
Note that $L$ is a random simple path with $L_{0} = x$ that satisfies that $\lim_{m \to \infty} |L_{m}| = \infty$, $\Pb^{x}$-almost surely. If $L$ is constructed from the SRW $S$ under $\Pb^{x}$, we will refer to it as the infinite loop-erased random walk (LERW) started at $x$. See \cite[Section 11]{LawLim} for further background.

\subsection{Constants}

We use $c$, $c'$, $C$, $C'$, $c_{1}$, $C_{1}, \dots$ to denote universal positive finite constants depending only on $d$, whose values may change between lines. If we want to emphasize that a constant depends on some parameter, we will use a subscript to indicate it. For example, $c=c_{\varepsilon}$ means the constant $c$ depends on $\varepsilon$. If $\{ a_n \}$ and $\{ b_n \}$ are two positive sequences, then we write $a_n = O(b_n)$ if there exists a constant $C \in (0, \infty)$ such that $a_{n} \le C b_{n}$ for all $n$. When we add a subscript to this $O$ notation, it means that the constant $C$ depends on the subscript. For instance, by $a_n = O_{\varepsilon} (b_n)$, we mean that $a_n \le C_{\varepsilon} b_n$.

\subsection{Simple random walk estimates}

Let $S$ be a simple random walk on $\mathbb{Z}^{d}$, and suppose $m$ and $n$ are real numbers such that $1 \le m < n $. Moreover, let $A = \{ x \in \mathbb{Z}^{d} \::\: m \le |x| \le n \}$, and set $\tau = \tau^{S}_{A^{c}}$ to be the first time that $S$ exits $A$. Then \cite[Proposition 1.5.10]{Lawb} gives that, for all $x \in A$,
\begin{equation}\label{srwbound}
\Pb^{x} \left( |S_{\tau} | \le m \right) = \frac{ |x|^{2-d} - n^{2-d}  + O (m^{1-d} )  }{m^{2-d} - n^{2-d}}.
\end{equation}
Whilst this approximation is good for large $m$, in this paper, we also need to consider the situation when $m=1$ and $|x|$ is large. In this case, $|S_{\tau} | \le m$ if and only if $S_{\tau} = 0$, and the estimate \eqref{srwbound} is not useful due to the $O (m^{1-d} )$ term. However, adapting the argument used to prove \cite[Proposition 1.5.10]{Lawb}, it is possible to establish that there exists a universal constant $a = a_{d} > 0$ such that
\begin{equation}\label{srwbound2}
\Pb^{x} \left( S_{\tau} = 0 \right) = \frac{ a |x|^{2-d} - a n^{2-d}  + O (|x|^{1-d} )  }{G (0) - a n^{2-d}},
\end{equation}
where $G (0)$, as defined by
\[G (0) = \sum_{j = 0}^{\infty} \Pb ( S_{j} = 0 ),\]
is the expected number of returns of $S$ to its starting point, which is finite in the dimensions we are considering.

In this article, we will also make use of another basic estimate for simple random walk on $\mathbb{Z}^d$, which is often called the gambler's ruin estimate. We take $\theta\in \mathbb{R}^d$ with $|\theta|=1$ and set $\widehat{S}_j=S_j\cdot\theta$. Let $\eta_n=\min\{j\ge 0\mathrel{:}\widehat{S}_j\le 0 \mbox{ or } \widehat{S}_j\ge n\}$. We denote by $\widehat{\mathbf{P}}^x$ the law of $\widehat{S}$ with starting point $x\in\mathbb{R}$. Then \cite[Proposition 5.1.6]{LawLim} guarantees that there exist $0<\alpha_1<\alpha_2<\infty$ such that: for all $1\le m\le n$,
\begin{equation}\label{gambler}
\alpha_1\frac{m+1}{n}\le \widehat{\mathbf{P}}^m(\widehat{S}_{\eta_n}\ge n)\le \alpha_2\frac{m+1}{n}.
\end{equation}
The gambler's ruin estimate gives upper and lower bounds on the probability that a simple random walk on $\mathbb{Z}^d$ projected onto a line escapes from one of the endpoints of a line segment.

\section{Loop-erased random walk estimates}\label{sec3}

The aim of this section is to prove Theorem \ref{thm:main1}. Due to the diffusive scaling of the LERW, it is convenient to reparameterise the result. In particular, we will prove the following, which clearly implies Theorem \ref{thm:main1}. Throughout this section, for $x\in\mathbb{Z}^d$, we write $\tau_{x} = \tau_{x}^{L}$ for the first time that the LERW $L$ hits $x$.

\begin{propn}\label{Lprop}
There exist constants $c_1,c_2\in(0,\infty)$ such that for every $x\in\mathbb{Z}^d\backslash\{0\}$ and $M>0$,
\[\mathbf{P}\left(\tau_{x}\in\left[M|x|^2,2M|x|^2-1\right]\right)\leq c_1\left(M|x|^2\right)^{1-d/2}\exp\left(-\frac{c_2}M\right).\]
Moreover, there exist constants $c_3,c_4,c_5,c_6\in(0,\infty)$ such that for every $x\in\mathbb{Z}^d\backslash\{0\}$ and $M\geq |x|^{-1}$,
\[\mathbf{P}\left(\tau_{x}\in\left[c_3 M|x|^2,c_4 M|x|^2\right]\right)\geq c_5\left(M|x|^2\right)^{1-d/2}\exp\left(-\frac{c_6}M\right).\]
\end{propn}
\bigskip

We will break the proof of this result into four pieces, distinguishing the cases $M\in(0,1)$ and $M\geq 1$, and considering the upper and lower bounds separately. See Propositions \ref{upper-M-large}, \ref{upper-M-small}, \ref{sw-prop} and \ref{lower-M-large} for the individual statements.

\subsection{Upper bound for $M\geq 1$}

The aim of this section is to establish the following, which is the easiest to prove of the constituent results making up Proposition \ref{Lprop}.

\begin{prop}\label{upper-M-large}
There exist constants $c_1,c_2\in(0,\infty)$ such that for every $x\in\mathbb{Z}^d\backslash\{0\}$ and $M\geq 1$,
\[\mathbf{P}\left(\tau_{x}\in\left[M|x|^2,2M|x|^2-1\right]\right)\leq c_1\left(M|x|^2\right)^{1-d/2}\exp\left(-\frac{c_2}M\right).\]
\end{prop}
\begin{proof}
Recalling the definition of $(\sigma_i)_{i\geq 0}$ from \eqref{sigmadef}, we have that
\begin{eqnarray*}
\mathbf{P}\left(\tau_{x}\in\left[M|x|^2,2M|x|^2-1\right]\right)&=&\sum_{i=\lceil M|x|^2\rceil}^{\lfloor 2M|x|^2-1\rfloor}\mathbf{P}\left(S_{\sigma_i}={x}\right)\\
&\leq&\mathbf{E}\left(\#\left\{i\geq \lceil M|x|^2\rceil:\:S_{\sigma_i}={x}\right\}\right).
\end{eqnarray*}
Using that $\sigma_i\geq i$, this implies that
\begin{eqnarray*}
\mathbf{P}\left(\tau_{x}\in\left[M|x|^2,2M|x|^2-1\right]\right)&\leq &\mathbf{E}\left(\#\left\{n\geq \lceil M|x|^2\rceil:\:S_{n}={x}\right\}\right)\\
&=&\sum_{n=\lceil M|x|^2\rceil}^\infty \mathbf{P}(S_n=x)\\
&\leq &\sum_{n=\lceil M|x|^2\rceil}Cn^{-d/2}\\
&\leq & C\left(M|x|^2\right)^{1-d/2},
\end{eqnarray*}
where for the second inequality, we have applied the upper bound on the transition probabilities of $S$ from \cite[Theorem 6.28]{Barbook}. Since it also holds that $\exp(-c_2/M)\geq C$ uniformly over $M\geq 1$, the result follows.
\end{proof}

\subsection{Upper bound for $M \in (0,1)$}

We will give an upper bound on the probability that $\tau_{x}$ is much smaller than $|x|^{2}$. More precisely, the goal of this section is to prove the following proposition. Replacing $M$ by $2M$, this readily implies the relevant part of Proposition \ref{Lprop}.

\begin{prop}\label{upper-M-small}
There exist constants $c_1,c_2\in(0,\infty)$ such that for every $x\in\mathbb{Z}^d\backslash\{0\}$ and $M\in(0,2)$,
\begin{equation}\label{upper-small}
\Pb \left( \tau_{x} \le M |x|^{2} \right) \le c_1\left(M|x|^2\right)^{1-d/2}\exp\left(-\frac{c_2}M\right).
\end{equation}
\end{prop}
\bigskip

Before diving into the proof, we observe that it is enough to show \eqref{upper-small} only for the case that both $|x|^{-1}$ and $M$ are sufficiently small. To see this, suppose that there exist some $c_1,c_2 \in (0, \infty)$ and $r_0\in (0,1)$ such that the inequality \eqref{upper-small} holds with constants $c_1,c_2$ for all $x$ and $ M$ satisfying $|x|^{-1} \vee M \le r_{0}$. The remaining cases we need to consider are (i) $|x|^{-1} \ge r_{0}$ and (ii) $M \in [r_{0}, 2)$. We first deal with case (i). If we suppose that $|x|^{-1} \ge r_{0}$ and $M < r_{0}^{2}$, then $M |x|^{2} < 1$, and so the probability on the left-hand side of \eqref{upper-small} is equal to $0$. On the other hand, if $|x|^{-1} \ge r_{0}$ and $M \in [r_{0}^{2}, 2)$, by choosing the constant $c_1$ to satisfy $c_1 \ge 2^{d/2-1}r_{0}^{2-d} \exp \{ c_2 r_{0}^{-2} \}$, we can ensure the right-hand side of \eqref{upper-small} is greater than 1, and so the inequality \eqref{upper-small} also holds in this case. Let us move to case (ii). We note that the probability on the left-hand side of \eqref{upper-small} can be always bounded above by
\[ \Pb \left( \tau^{S}_{x} < \infty \right) \le C|x|^{2-d}\]
for some constant $C \in (0, \infty)$, where we have applied \eqref{srwbound2} to deduce the inequality. Thus, choosing the constant $c_1$ so that $c_1\ge C 2^{d/2-1}\exp \{ c_2 r_{0}^{-1} \}$, the inequality \eqref{upper-small} holds. Consequently, replacing the constant $c_1$ by $c_1\vee 2^{d/2-1}r_{0}^{2-d} \exp \{ c_2 r_{0}^{-2} \} \vee C 2^{d/2-1}\exp \{ c_2 r_{0}^{-1} \}$, the inequality \eqref{upper-small} holds for all $x \in \mathbb{Z}^{d} \setminus \{ 0 \}$ and $M \in (0, 2)$.

We next give a brief outline of the proof of Proposition \ref{upper-M-small}, assuming that both $|x|^{-1}$ and $M$ are sufficiently small. We write $A_{x} = B_{\infty} ( 0, |x|/4 \sqrt{d})$ for the box of side length $|x|/2 \sqrt{d}$ centered at the origin, and let
\[t_{x} = \tau^{L}_{A_{x}^{c}}\]
be the first time that $L$ exits $A_{x}$. Note that $x\not\in A_x$, and so
\begin{align}\label{ds1}
\Pb \left( \tau_{x} \le M |x|^{2} \right) \le \Pb \left( t_{x} \le \tau_{x} \le M |x|^{2} \right)  \le \Pb \left( \tau_{x} < \infty \:\vline\: t_{x} \le M |x|^{2}   \right) \, \Pb \left( t_{x} \le M |x|^{2} \right).\end{align}
Writing
\begin{equation}\label{ds1'}
p_{x, M} = \Pb \left( \tau_{x} < \infty \:\vline\: t_{x} \le M |x|^{2}   \right) \ \ \ \text{ and }  \ \ \ q_{x, M} = \Pb \left( t_{x} \le M |x|^{2} \right),
\end{equation}
we will show that
\[p_{x, M} \le C |x|^{2-d},\qquad q_{x, M} \le C \exp \{ c M^{-1} \}\]
in Lemmas \ref{ds-lem-1} and \ref{ds-lem-2} below, respectively. Proposition \ref{upper-M-small} is then a direct consequence of these lemmas.

We start by dealing with $p_{x, M}$, as defined in \eqref{ds1'}.

\begin{lem}\label{ds-lem-1}
There exists a constant $C \in (0, \infty)$ such that for all $x \in \mathbb{Z}^{d} \setminus \{ 0 \}$ and $M \in (0,2)$ with $\Pb ( t_{x} \le M |x|^{2} ) > 0$,
\[p_{x, M} \le C |x|^{2-d}.\]
\end{lem}

\begin{proof}
Let
\[\Lambda = \left\{ \lambda \::\: \Pb \left( t_{x} \le M |x|^{2},\:L [0, t_{x} ] = \lambda \right) > 0 \right\}\]
be the set of all possible paths for $L [0, t_{x} ] $ satisfying $ t_{x} \le M |x|^{2}$. For $\lambda \in \Lambda$, we write $R = R^{\lambda}$ for a random walk conditioned on the event that $R [1, \infty ) \cap \lambda = \emptyset$. Note that $R$ is a Markov chain (see \cite[Section 11.1]{LawLim}).  We use $\mathbf{P}^{y}_{R}$ to denote the law of $R$ started from $R_{0} = y$. Then the domain Markov property for $L$ (see \cite[Proposition 7.3.1]{Lawb}) ensures that
\[p_{x, M} = \frac{\sum_{\lambda \in \Lambda} \mathbf{P}^{\lambda_{\text{len} (\lambda)}}_{R} \left( x \in \text{LE} \left( R [0, \infty ) \right) \right) \Pb ( L [0, t_{x} ] = \lambda ) }{q_{x, M}} \le \max_{\lambda \in \Lambda} \mathbf{P}^{\lambda_{\text{len} (\lambda)}}_{R} \left( x \in R [0, \infty )  \right).\]
Therefore, it suffices to show that there exists a constant $C \in (0, \infty)$ such that for all $x \in \mathbb{Z}^{d} \setminus \{ 0 \}$, $M \in (0, 2)$ with $\Pb ( t_{x} \le M |x|^{2} ) > 0$ and $\lambda \in \Lambda$,
\[\mathbf{P}^{\lambda_{\text{len} (\lambda)}}_{R} \left( x \in R [0, \infty )  \right) \le C |x|^{2-d}.\]

With the above goal in mind, let us fix $\lambda \in \Lambda$. We set $u :=\tau^{R}_{ B ( |x| / 2 )^{c} }$ for the first time that $R$ exits $B (|x|/ 2)$. (Note that $A_{x} \subseteq B (|x| / 2)$ by our construction.) Using the strong Markov property for $R$ at time $u$, we have
\[\mathbf{P}^{\lambda_{\text{len} (\lambda)}}_{R} \left( x \in R [0, \infty )  \right) \le \max_{y \in \partial B (|x|/2) } \mathbf{P}^{y}_{R} \left( x \in R [0, \infty )  \right).\]
On the other hand, it follows from \eqref{srwbound} that
\[\min_{ y \in \partial B (|x|/2) } \Pb^{y} \left( S [0, \infty ) \cap \lambda = \emptyset \right) \ge \min_{ y \in \partial B (|x|/2) } \Pb^{y} \left( S [0, \infty ) \cap A_{x} = \emptyset \right) \ge c_{0} \]
for some constant $c_{0} > 0$. Combining these estimates and  using \eqref{srwbound2}, we see that, for each $y \in \partial B (|x| / 2)$,
\[\mathbf{P}^{y}_{R} \left( x \in R [0, \infty )  \right) \leq \frac{\Pb^{y} \left( x \in S [0, \infty ) \right)}{\Pb^{y} \left( S [0, \infty ) \cap \lambda = \emptyset \right)} \le \frac{1}{c_{0}}  \Pb^{y} \left( x \in S [0, \infty ) \right) \le C |x|^{2-d}\]
for some constant $C \in (0, \infty)$. This finishes the proof.
\end{proof}

Recall that $q_{x, M}$ was defined at \eqref{ds1'}. We will next estimate $q_{x, M}$ as follows.

\begin{lem}\label{ds-lem-2}
There exist constants $c, C \in (0, \infty)$ such that for all $x \in \mathbb{Z}^{d} \setminus \{ 0 \}$ and $M \in (0, 2)$,
\begin{equation}\label{ds4}
q_{x, M} \le C \exp \{ -c M^{-1} \}.
\end{equation}
\end{lem}
\begin{proof}
As per the discussion after \eqref{upper-small}, it suffices to prove \eqref{ds4} only in the case that both $|x|^{-1}$ and $M$ are sufficiently small. In particular, throughout the proof, we assume that
\begin{equation}\label{ds-4-2}
M \le (3200  d)^{-1}.
\end{equation}
Furthermore, we may assume
\begin{equation}\label{ds-4-1}
| x | M \ge (4 \sqrt{d} )^{-1},
\end{equation}
since  $q_{x, M} = 0$ when $ | x | M < (4 \sqrt{d} )^{-1}$. (Notice that it must hold that $t_{x} \ge |x| (4 \sqrt{d} )^{-1}$.)

Now, define the increasing sequence of boxes $\{ A^{i} \}_{i = 1 }^{N}$, where $N = \lfloor (1600   d  M)^{-1} \rfloor$, by setting
\[A^{i} = B_{\infty} \left( 0,  400 \sqrt{d} \, |x| M i  \right)\]
for $1 \le i \le N$. Observe that the particular choice of $N$ ensures that $A^{N} \subseteq A_{x} = B_{\infty} ( 0, |x|/4 \sqrt{d})$, and the assumption \eqref{ds-4-2} guarantees that
\begin{equation}\label{ds-4-3}
N \ge (3200  d M)^{-1}.
\end{equation}
Also, we note that  $\text{dist} ( \partial A^{i-1}, \partial A^{i})$ is bigger than $400 \sqrt{d} \, |x| M -1$, which is in turn bounded below by $99$ because of \eqref{ds-4-1}. As a consequence, it is reasonable to compare the number of lattice points in the set $ L \cap (A^{i} \setminus A^{i-1} ) $ with $|x|^{2} M^{2}$. To be more precise, let $t^{0} = 0$, and, for $ i \ge 1$, set
\[t^{i} = \tau^{L}_{(A^{i})^{c}}\]
to be the first time that $L$ exits $A^{i}$. Then \cite[Corollary 3.10]{BJ} shows that there exists a deterministic constant $c_{1}\in(0,1)$ such that for all $x \in \mathbb{Z}^{d} \setminus \{ 0 \}$ and  $M \in (0, 2)$ satisfying \eqref{ds-4-2} and \eqref{ds-4-1},
\begin{equation}\label{ds5}
\Pb \left( t^{i} - t^{i-1} \ge c_{1} |x|^{2} M^{2} \:\vline\: L [0, t^{i-1} ] \right) \ge c_{1}, \qquad 1 \le i \le N.
\end{equation}
With the inequality \eqref{ds5} and a constant $a>0$ satisfying
\begin{equation}\label{ds6}
2 \sqrt{\frac{2 a}{1- c_{1}}} < \frac{1}{6400  d} \log \frac{1}{1- c_{1}},
\end{equation}
it is possible to apply \cite[Lemma 1.1]{BB} to deduce the result of interest. In particular, the following table explains how the quantities of this article are substituted into \cite[Lemma 1.1]{BB}.
\begin{center}
\begin{tabular}{r||c|c|c|c|c|c}
\cite[Lemma 1.1]{BB} &  $X$ & $Y_{i}$ &  $n$ & $p$ & $b$ & $x$ \\\hline
This article & $t_{x}$ &$t^{i} - t^{i-1}$ & $N$ & $1 - c_{1}$ & $2 \, |x|^{-2} M^{-2}$ & $a M |x|^{2}$
\end{tabular}
\end{center}
Then, from \cite[Lemma 1.1]{BB}, one has
\begin{align*}
\Pb \left( t_{x} \le a M |x|^{2} \right) &{\le} \exp \left\{  2 M^{-1} \sqrt{\frac{2 a}{1- c_{1}}}   -  N \log \frac{1}{1- c_{1}}  \right\}  \\
&{\le} \exp \left\{  \left( 2   \sqrt{\frac{2 a}{1- c_{1}}}   - \frac{1}{3200d} \log \frac{1}{1- c_{1}}  \right)   M^{-1}   \right\}  \\
&{\le} \exp \left\{ - \frac{M^{-1} }{6400 d}  \log \frac{1}{1- c_{1}} \right\},
\end{align*}
where for the second and third inequalities, we apply \eqref{ds-4-3} and \eqref{ds6}, respectively. Rewriting $a M = M'$ completes the proof.
\end{proof}

\begin{proof}[Proof of Proposition \ref{upper-M-small}]
Proposition \ref{upper-M-small} follows directly from \eqref{ds1} and Lemmas \ref{ds-lem-1} and \ref{ds-lem-2} (and the basic observation that $M^{1-d/2}\geq 2^{1-d/2}$ for $M\in (0,2)$).
\end{proof}

\subsection{Lower bound for $M\in(0,1)$}

Recall that for $x\in\mathbb{Z}^d\setminus\{0\}$, $\tau_x$ indicates the first time that $L$ hits $x$. The aim of this section is to bound below the probability that $\tau_x$ is much smaller than $|x|^2$. In particular, the following is the main result of this section, which readily implies the part of the lower bound of Proposition \ref{Lprop} with $|x|^{-1}\leq M<1$.

\begin{prop}\label{sw-prop}
There exist constants $c,c',R\in(0,\infty)$ such that for every $x\in\mathbb{Z}^d\backslash\{0\}$ and $|x|^{-1}\le M<1$,
\begin{equation}\label{sw-0}
	\mathbf{P}\left(\tau_x\in[R^{-1}M|x|^2,RM|x|^2]\right)\ge c'(M|x|^2)^{1-d/2}\exp\left(-\frac{c}{M}\right).
\end{equation}
\end{prop}
\bigskip

Before moving on to the proof, we will first show that once we prove that there exists a constant $n_0\geq1$ such that (\ref{sw-0}) holds for $n_0|x|^{-1}\leq M<1$,
we obtain (\ref{sw-0}) for every $x\in \mathbb{Z}^d\setminus \{0\}$ and $|x|^{-1}\le M<1$ by adjusting $c$, $c'$ and $R$ as needed. Let us consider the following three events:
\begin{itemize}
	\item $S[0,\tau^S_x]$ is a simple path of length $\lceil M|x|^2\rceil$,
	\item $S[\tau^S_x+1,\tau^S_{B(0,2r)^c}]$ is a simple path that does not intersect $S[0,\tau^S_x]$,
	\item $S[\tau^S_{B(0,2r)^c},\infty)\cap B(0,\frac{3}{2}|x|)=\emptyset$,
\end{itemize}
where we set $r=|x|\vee n_0$. It is straightforward to see that $\tau_x\in[M|x|^2,2M|x|^2]$ holds on the intersection of these events. By constructing a simple random walk path that satisfies the first two conditions up to the first exit time from the Euclidean ball $B(0,2r)$ and then applying (\ref{srwbound}) and the strong Markov property, we have
\[\mathbf{P}\left(\tau_x\in[M|x|^2,2M|x|^2]\right)\ge a(2d)^{-2 M|x|^2}(2d)^{-2r},\]
for some $a>0$ that does not depend on $M$ or $x$. Suppose $1\leq M|x|\leq n_0$. If $|x|< n_0$, then the right-hand side is bounded below as follows:
\[a(2d)^{-2 M|x|^2}(2d)^{-2r}\ge a(2d)^{-(2n_0^2+2n_0)}\geq C\geq c'(M|x|^2)^{1-d/2}e^{-\frac{c}{M}}.\]
On the other hand, if $|x|\ge n_0$, then the right-hand side satisfies
\[a(2d)^{-2 M|x|^2}(2d)^{-2r}\ge C\left( M|x|^2\right)^{1-d/2}e^{-cM|x|^2-c'|x|}\geq C\left( M|x|^2\right)^{1-d/2}e^{-c''/M}.\]
In particular, by replacing $R$ by $R\vee 2$ and adjusting $c$, $c'$ appropriately, the result at \eqref{sw-0} can be extended to $1\leq M|x|<|x|$.

The structure of this section is as follows. First, we define several subsets of $\mathbb{R}^d$. These will be used to describe a number of events involving the simple random walk $S$ whose loop-erasure is $L$. We provide some key estimates on the probabilities of these events in Lemmas \ref{sw-lem-2} and Lemma \ref{sw-lem-1}. Finally, applying these results, we prove Proposition \ref{sw-prop} at the end of this section.

We begin by defining ``a tube connecting the origin and $x$'', which will consist of a number of boxes of side-length $M|x|$. To this end, for $M\in(0,1)$, let
\[N_M=\left\lceil \frac{1}{M}+\frac{1}{2}\right\rceil.\]
Moreover, for $x=(x^1,\dots, x^d)\in\mathbb{Z}^d\setminus \{0\}$ and $M\in(0,1)$, define a sequence $\{b_i\}$ of vertices of $\mathbb{R}^d$ by setting
\begin{equation}\label{sw-1}
b_i=\left(iMx^1,\dots,iMx^d\right)
\end{equation}
for  $i\in\{0,1,\dots,2N_M\}$. Let us consider a rotation around the origin that aligns the $x^1$-axis with the line through the origin and $x$. We denote by $B$ and $Q$ the images of $[-{M|x|}/{2},{M|x|}/{2}]^d$ and $\{0\}\times [-{M|x|}/{2},{M|x|}/{2}]^{d-1}$ under this rotation, respectively. For $y=(y^1,\dots,y^d)\in\mathbb{R}^d$, we let
\begin{equation}\label{sw-2}
	\widetilde{B}(y,r)=\left\{\left(y^1+rz^1,\dots,y^d+rz^d \right)\mathrel{:}(z^1,\dots,z^d)\in B\right\},
\end{equation}
be the tilted cube of side-length $rM|x|$ centered at $y$, and we write $B_i$ for $\widetilde{B}(b_i,1)$. For $i=0,1,\dots,2N_M$ and $a,b\in\mathbb{R}$ with $a<b$, also let
\[Q(y,r)=\left\{(y^1+rz^1,\dots,y^d+rz^d)\mathrel{:}(z^1,\dots,z^d)\in Q\right\},\]
\[B_i[a,b]=\left\{\left(z^1+sMx^1,\dots,z^d+sMx^d\right)\mathrel{:}s\in [a,b],\ (z^1,\dots,z^d)\in Q_i(0)\right\},\]
where
\[Q_i(b)=Q\left(\left((i-\frac{1}{2}+b)M|x|,\dots,(i-\frac{1}{2}+b)M|x|\right),1\right).\]
We also set
\[\widetilde{Q}_i(b)=Q\left(\left((i-\frac{1}{2}+b)M|x|,\dots,(i-\frac{1}{2}+b)M|x|\right),\frac{1}{2}\right),\]
and note that, by definition, $\widetilde{Q}_i(b)\subseteq Q_i(b)$ for all $i\ge 0$ and $b\in\mathbb{R}$. Observe that every $\widetilde{Q}_i(b)$ is perpendicular to the line through the origin and $x$, and that $Q_i\coloneqq Q_i(0)$ is the ``left face'' of the cube $B_i=B_i[0,1]$. Finally, we write
\begin{equation}\label{sw-4}
	\widetilde{Q}_i\coloneqq\widetilde{Q}_i(0),\qquad i=1,2,\dots, 2N_M+1,
\end{equation}
and set $Q_0=\widetilde{Q}_0=\{0\}$ for convenience. See Figure \ref{sw-fig-1} for a graphical representation of the situation.

\begin{figure}[t]
		\centering
		\includegraphics[width=0.9\textwidth]{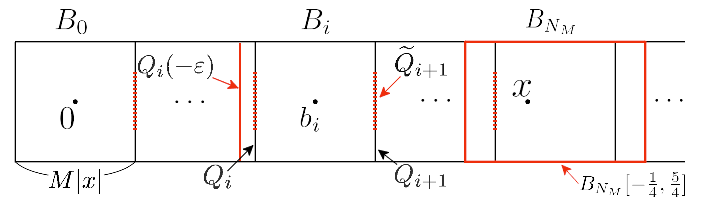}
		\caption{Illustration of $B_i$, $Q_i$ and $\widetilde{Q}_i$ for a given $x$.}\label{sw-fig-1}
\end{figure}

In this section, it will be convenient to consider $S$ (recall that $L=\mathrm{LE}(S[0,\infty))$) as a continuous curve in $\mathbb{R}^d$ by linear interpolating between discrete time points, and thus we may assume that $S(k)$ is defined for all non-negative real $k$. If $\lambda$ is a continuous path in $\mathbb{R}^d$ and $A\subseteq \mathbb{R}^d$, we write
\[\tau^\lambda (A)=\inf\left\{t\ge 0\::\:\lambda(t)\in A \right\},\]
and also, for $x\in\mathbb{R}^d$, we set $\tau^\lambda_x=\tau^\lambda(\{x\})$, analogous to the notation of first hitting times for discrete paths (\ref{hitting}).

In order to obtain the lower bound (\ref{sw-0}), we consider events under which the LERW $L$, started at the origin, travels through the ``tube'' $\bigcup_{i=0}^{N_M} B_i$ until it hits $x$. See Figure \ref{sw-fig-4} for a graphical representation.

\begin{dfn}\label{sw-def-1}
We define the events $F_i$, $i=0,1,\dots,2N_M$, as follows. Firstly,
\[F_0=\left\{\begin{array}{c}
\tau^S(Q_1)<\infty,\: S(\tau^S(Q_1))\in \widetilde{Q}_1,\: S[0,\tau^S(Q_1)]\subset B_0,\\
	S[\tau^S(Q_1(-\varepsilon)),\tau^S(Q_1)]\cap Q_1(-2\varepsilon)=\emptyset
\end{array}\right\}.\]
For $i=1,2,\dots, N_M-1$,
\[F_i=\left\{\begin{array}{c}
\tau^S(Q_i)<\tau^S(Q_{i+1})<\infty,\: S(\tau^S(Q_{i+1}))\in \widetilde{Q}_{i+1},\: S[\tau^S(Q_i),\tau^S(Q_{i+1})]\subset B_{i}[-\varepsilon,1],\\
S[\tau^S(Q_{i+1}(-\varepsilon)),\tau^S(Q_{i+1})]\cap Q_{i+1}(-2\varepsilon)=\emptyset
\end{array}\right\}.\]
Moreover,
\[	F_{N_M}=\left\{\begin{array}{c}
\tau^S(Q_{N_M})< \tau^S_x<\tau^S\left(Q_{N_M+1}\left(\frac{1}{4}\right)\right)<\infty,\\
S\left(\tau^S\left(Q_{N_M+1}\left(\frac{1}{4}\right)\right)\right)\in\widetilde{Q}_{N_M+1}\left(\frac{1}{4}\right),\: S[\tau^S(Q_{N_M}),\tau^S(Q_{N_M+1})]\subset B_{N_M}\left[-\frac{1}{4},\frac{5}{4}\right],\\
S\left[\tau^S_x,\tau^S\left(Q_{N_M+1}\left(\frac{1}{4}\right)\right)\right]\cap\mathrm{LE}(S[\tau^S(Q_{N_M}),\tau^S_x])=\emptyset
\end{array}\right\},\]
\[F_{N_M+1}=\left\{\begin{array}{c}
\tau^S\left(Q_{N_M+1}\left(\frac{1}{4}\right)\right)<\tau^S(Q_{N_M+2}),\:\tau^S(Q_{N_M+2})\in\widetilde{Q}_{N_M+2},\\
S\left[\tau^S\left(Q_{N_M+1}\left(\frac{1}{4}\right)\right),\tau^S(Q_{N_M+2})\right]\subset B_{N_M+1}\left[\frac{1}{4}-\varepsilon,1\right]
\end{array}\right\},\]
and, for $i=N_M+1,\dots,2N_M$,
\[F_i=\left\{\begin{array}{c}
\tau^S(Q_i)<\tau^S(Q_{i+1})<\infty,\: S(\tau^S(Q_{i+1}))\in \widetilde{Q}_{i+1},\\
S[\tau^S(Q_i),\tau^S(Q_{i+1})]\subset B_{i}[-\varepsilon,1]
\end{array}\right\}.\]
\end{dfn}

\begin{figure}[t]
\centering
\includegraphics[width=0.9\textwidth]{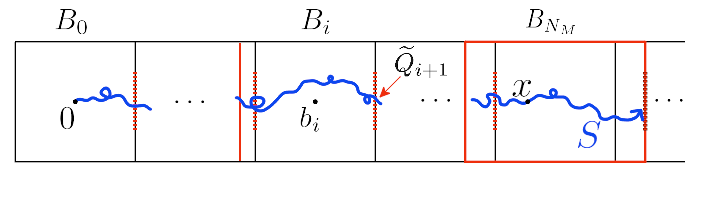}
\vspace{-15pt}
\caption{Illustration of the events $F_i$, $i=0,1,\dots, N_M$.}\label{sw-fig-4}
\end{figure}

The first three conditions of the definition of each $F_i$, $i=0,1,\dots,2N_M$, require that $S$ travels inside the ``tube'' and it exits each $B_i$ at a point which is not close to $\partial Q_{i+1}$. Furthermore, the last condition in the definition of $F_0$, the last two conditions in that of each $F_i$, $i=1,2,\dots,N_M-1$, and the third condition in that of $F_{N_M}$ control the range of backtracking of $S$. Finally, the last condition in the definition of $F_{N_M}$ and events $F_i$, $i=N_M+1,\dots,2N_M$ guarantee that $x$ remains in $\mathrm{LE}(S[0,\tau^S(Q_{2N_M+1})])$.

Next, we define events that provide upper and lower bounds for the length of the loop-erasure of $S$ in each $B_i$. For $i\in\{0,1,\dots,N_M-1\}$, let
\begin{equation}\label{sw-def-xi}
	\xi_i=\mathrm{LE}\left(S[0,\tau^S(Q_{i+1})]\right),
\end{equation}
and also set $\xi_{N_M}=\mathrm{LE}(S[0,\tau^S_x])$. We further define
\begin{align}
\lambda_i&=\mathrm{LE}\left(S[\tau^S(Q_i),\tau^S(Q_{i+1})]\right),\qquad 1\le i\le N_M-1,	\notag	\\
\lambda_{N_M}&=\mathrm{LE}\left(S[\tau^S(Q_{N_M}),\tau^S_x]\right),	\notag	\\
\xi'_0&=\xi_0,\notag\\
\xi'_i&=\xi_0\oplus \lambda_1\oplus \cdots \oplus \lambda_i,\qquad i\ge 1.	\label{sw-def-lambda}
\end{align}
Since $\xi_i$ is not necessarily a simple curve, $\xi_i$ and $\xi'_i$ do not coincide in general. However, the restriction on the backtracking of $S$ on the events $F_i$ and a cut time argument (see Definition \ref{sw-def-4} and Definition \ref{sw-def-5} below) enable us to handle the difference between $\xi_i$ and $\xi'_i$. This will be discussed later, in the proof of Proposition \ref{sw-prop}.

We now define events upon which the length of $\lambda_i$ is bounded above.

\begin{dfn}\label{sw-def-3}
For $C>0$, the event $G_0(C)$ is given by
\[G_0(C)=\left\{\mathrm{len}(\xi_0)\le CM^2|x|^2\right\},\]
and for $i=1,2,\dots,N_M$, the event $G_i(C)$ is given by
\[G_i(C)=\left\{\mathrm{len}(\lambda_i)\le CM^2|x|^2\right\}.\]
\end{dfn}

In the following lemma, we demonstrate that $G_i$ occurs with high conditional probability. Recall that $\widetilde{Q}_i$ was defined at (\ref{sw-4}).

\begin{lem}\label{sw-lem-2}
For any $\delta>0$, there exists a constant $C_+>0$ such that
\begin{equation}\label{sw-lem2-1}
\mathbf{P}^z\left(G_i(C_+)\:\vline\:F_i\right)\ge 1-\delta,
\end{equation}
uniformly in $x\in\mathbb{Z}^d\backslash\{0\}$, $|x|^{-1}\le M<1$, $i\in \{0,1,\dots,N_M\}$ and $z\in \widetilde{Q}_i$.
\end{lem}

\begin{proof}
For $i=0,1,\dots,N_M-1$, $y\in B_i[-\varepsilon,1]$ and $z\in \widetilde{Q}_i$, we have that
\[\mathbf{P}^z\left(y\in \lambda_i\mid F_i\right)\leq\mathbf{P}^{z}(\tau^S_y<\tau^S(Q_{i+1})\mid F_i).\]
Moreover, by \eqref{gambler} and translation invariance, we have that there exists some constant $c>0$ such that
\begin{equation}\label{sw-lem2-1.1}
\inf_{z\in \widetilde{Q}_i}\mathbf{P}^z(F_i)\ge c\varepsilon,
\end{equation}
for all $i=1,2,\dots,N_M-1$. For $i=0$, the same argument yields that $\mathbf{P}(F_0)\ge c\varepsilon$. Thus, it follows from (\ref{srwbound2}) and (\ref{sw-lem2-1.1}) that
\[\mathbf{P}^{z}(\tau^S_y<\tau^S(Q_{i+1})\mid F_i)\le \frac{\mathbf{P}^{z}(\tau^S_y<\tau^S(Q_{i+1}))}{\mathbf{P}^z(F_i)}	\le C\mathbf{P}^z(\tau^S_y<\infty)\le C\left(|y-z|^{2-d}\wedge 1\right),\]
for some constant $C>0$. By taking the sum over $y\in B_i[-\varepsilon,1]$, we have that
\begin{equation}\label{sw-lem2-2}
\mathbf{E}^z(\mathrm{len}(\lambda_i)\mid F_i)=\sum_{y\in B_i[-\varepsilon,1]}\mathbf{P}^z(y\in \lambda_i\mid F_i)\le C\sum_{y\in B_i[-\varepsilon,1]}\left(|y-z|^{2-d}\wedge 1\right)\le CM^2|x|^2.\end{equation}
The same argument also applies to the case $i=0$, and thus we have
\begin{equation}\label{sw-lem2-2.3}
	\mathbf{E}^0(\mathrm{len}(\xi_0)\mid F_0)\le CM^2|x|^2.
\end{equation}
Similarly, for the case $i=N_M$, recalling that $\lambda_{N_M}$ was defined at (\ref{sw-def-lambda}), we have that
\begin{align}
{\mathbf{P}^{z}\left(y\in \lambda_{N_M}\mid F_{N_M}\right)}	&\le \mathbf{P}^z\left(\tau^S_y<\tau^S_x\mid F_{N_M}\right)	\notag	\\
	&\le\frac{\mathbf{P}^z\left(\tau^S_y<\tau^S_x<\tau^S(B_{N_M}[-\frac{1}{4},\frac{5}{4}]^c)\right)}{\mathbf{P}^z\left(F_{N_M}\right)}	\notag\\
	&=\frac{\mathbf{P}^z\left(\tau^S_y<\tau^S_x\wedge \tau^S(B_{N_M}[-\frac{1}{4},\frac{5}{4}]^c)\right)\mathbf{P}^y\left(\tau^S_x<\tau^S(B_{N_M}[-\frac{1}{4},\frac{5}{4}]^c)\right)}{\mathbf{P}^z\left(F_{N_M}\right)},	\label{sw-lem2-3}
\end{align}
where we used the strong Markov inequality to obtain the last inequality. We will prove that $\mathbf{P}^z(F_{N_M})\ge C'(M|x|)^{2-d}$ later in this subsection, see (\ref{sw-prop1-6}).

Now we bound above the sum of the numerator of (\ref{sw-lem2-3}) over $y\in B_{N_M}[-\frac{1}{4},\frac{5}{4}]$, separating into three cases depending on the location of $y$.
\begin{enumerate}
\item[(i)] For $y\in B(z,\frac{1}{18}M|x|)$, it follows from (\ref{srwbound2}) that
\begin{align*}
\mathbf{P}^z\left(\tau^S_y<\tau^S\left(B_{N_M}\left[-\frac{1}{4},\frac{5}{4}\right]^c\right)\wedge\tau^S_x\right)&=C(|y-z|^{2-d}-(M|x|/2)^{2-d})+O(|y-z|^{1-d}),	\\
	\mathbf{P}^y\left(\tau^S_x<\tau^S\left(B_{N_M}\left[-\frac{1}{4},\frac{5}{4}\right]^c\right)\right)&\le C(M|x|)^{2-d}.
\end{align*}
Thus we have that
\begin{align}
	\sum_{y\in B(z,\frac{1}{18}M|x|)} &\mathbf{P}^z\left(\tau^S_y<\tau^S_x\wedge\tau^S\left(B_{N_M}\left[-\frac{1}{4},\frac{5}{4}\right]^c\right)\right)\mathbf{P}^y\left(\tau^S_x<\tau^S\left(B_{N_M}\left[-\frac{1}{4},\frac{5}{4}\right]^c\right)\right)	\notag	\\
	&\le C(M|x|)^{2-d}\sum_{k=1}^{\frac{1}{18}M|x|}\sum_{|y-z|=k}(k^{2-d}-(M|x|/2)^{2-d}+O(k^{1-d}))	\notag	\\
	&\le C(M|x|)^{4-d}.	\notag
\end{align}
\item[(ii)] For $y\in B(x,\frac{1}{18}M|x|)$, a similar argument to case (i) yields that
\begin{align*}
	\sum_{y\in B(x,\frac{1}{18}M|x|)} &\mathbf{P}^z\left(\tau^S_y<\tau^S_x\wedge \tau^S\left(B_{N_M}\left[-\frac{1}{4},\frac{5}{4}\right]^c\right)\right)\mathbf{P}^y\left(\tau^S_x<\tau^S\left(B_{N_M}\left[-\frac{1}{4},\frac{5}{4}\right]^c\right)\right)	\\
	&\le C(M|x|)^{4-d}.
\end{align*}
\item[(iii)] For $y\in B_{N_M}[-\frac{1}{4},\frac{5}{4}]\setminus(B(z,\frac{1}{18}M|x|)\cup B(x,\frac{1}{18}M|x|))$, we have that
\begin{align*}
	\mathbf{P}^z\left(\tau^S_y<\tau^S_x<\tau^S\left(B_{N_M}\left[-\frac{1}{4},\frac{5}{4}\right]^c\right)\right)&\le \mathbf{P}^z(\tau^S_y<\infty)\le C(M|x|)^{2-d},	\\
	\mathbf{P}^y\left(\tau^S_x<\tau^S\left(B_{N_M}\left[-\frac{1}{4},\frac{5}{4}\right]^c\right)\right)&\le \mathbf{P}^y(\tau^S_x<\infty)\le C(M|x|)^{2-d},
\end{align*}
which gives
\begin{align}
	\sum_{\substack{y\in B_{N_M}[-\frac{1}{4},\frac{5}{4}]\\|y-z|,|y-x|\ge \frac{1}{18}M|x|}}&\mathbf{P}^z\left(\tau^S_y<\tau^S_x<\tau^S\left(B_{N_M}\left[-\frac{1}{4},\frac{5}{4}\right]^c\right)\right)\mathbf{P}^y\left(\tau^S_x<\tau^S\left(B_{N_M}\left[-\frac{1}{4},\frac{5}{4}\right]^c\right)\right)	\notag	\\
	&\le C\sum_{\substack{y\in B_{N_M}[-\frac{1}{4},\frac{5}{4}]\\|y-z|,|y-x|\ge \frac{1}{18}M|x|}}(M|x|)^{4-2d}\le C(M|x|)^{4-d}.	\notag
\end{align}
\end{enumerate}
Thus, by (\ref{sw-lem2-3}), it holds that
\begin{equation}\label{sw-lem2-4}
\mathbf{E}^z(\mathrm{len}(\lambda_{N_M})\mid F_{N_M})=\sum_{y\in B_{N_M}[-\frac{1}{4},\frac{5}{4}]}\mathbf{P}^z(y\in \lambda_{N_M}\mid F_{N_M})	\le \frac{C(M|x|)^{4-d}}{c'\cdot C'(M|x|)^{2-d}} \le CM^2|x|^2.
\end{equation}
Combining (\ref{sw-lem2-2}), (\ref{sw-lem2-2.3}) and (\ref{sw-lem2-4}) with Markov's inequality, it holds that
\[	\mathbf{P}^z(\mathrm{len}(\lambda_i)\ge C_+M^2|x|^2\mid F_i)\le C/C_+,\]
for $i=0,1,\dots,N_M$. By taking $C_+=\delta^{-1}C$, we obtain (\ref{sw-lem2-1}).
\end{proof}

Now we will deal with events involving $S$ upon which the length of $\lambda_i$ is bounded below and, at the same time, the gap between the lengths of $\xi_i$ and $\xi'_i$ is bounded above. First, we define a special type of cut time of $S$ in each $B_i$.

\begin{dfn}\label{sw-def-4}
A nice cut time in $B_i$ is a time $k$ satisfying the following conditions:
\begin{itemize}
	\item $k\in\left[\tau^S(Q_i(\varepsilon/3)),\tau^S(Q_i(\varepsilon))\right]$,
	\item $S(k)\in B(S(\tau^S(Q_i)),\varepsilon M|x|/2)$,
	\item $S[\tau^S(Q_i),k]\cap S\left[k+1,\tau^S(Q_i(\varepsilon))\right]=\emptyset$,
	\item $S[k+1,\tau^S(Q_i(\varepsilon))]\cap Q_i=\emptyset$.
\end{itemize}
\end{dfn}

Second, let $B'_i$ (respectively $B''_i$) be the cube of side-length $M|x|/3$ (respectively $M|x|/9$) centered at $b_i$ whose faces are parallel to those of $B_i$, i.e.
\[B'_i=\widetilde{B}\left(b_i,\frac{1}{3}\right),\qquad B''_i=\widetilde{B}\left(b_i,\frac{1}{9}\right),\]
where $b_i$ and $\widetilde{B}(y,r)$ are as defined in (\ref{sw-1}) and (\ref{sw-2}), respectively. We denote by $Q_i^L$ (respectively $Q_i^R$) the ``left (respectively right) face'' of $B'_i$. See Figure \ref{sw-fig-2}.

Let $\rho_i=\inf\{n\ge \tau^S(B''_i):\: S(k)\in (B'_i)^c\}$ be the first time that $S$ exits $B'_i$ after it first entered $B''_i$. We define a set of local cut times of $S$ by
\[K_i=\left\{\tau^S(B''_i)\le k\le \rho_i \mathrel{:}S(k)\in B''_i,\: S[\tau^S(B'_i),k]\cap S[k+1,\rho_i]=\emptyset\right\}.\]

Finally, we define events $H_i^{(j)}~(j=1,2,3,4)$ as follows.

\begin{figure}[t]
\centering
\includegraphics[width=0.45\textwidth]{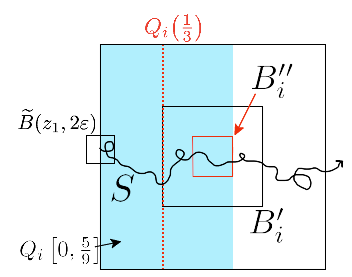}
\caption{Illustration of $B'_i$ and $B''_i$.}\label{sw-fig-2}
\end{figure}

\begin{dfn}\label{sw-def-5} For $1\le i\le N_M-1$ and $l>0$,
\begin{align}
	H_i^{(1)}&=\left\{S \mbox{ has a nice cut time in $B_i$}\right\}	\cap\left\{0<\tau^S(Q_i(\varepsilon))-\tau^S(Q_i)\le C\varepsilon^2 M^2|x|^2\right\},	\notag\\
	H_i^{(2)}&=\left\{\begin{array}{c}
\tau^S(B'_i)<\tau^S(Q_{i+1}),\:S(\tau^S(Q_i(1/3)))\in Q_i^L,\\
		S[\tau^S(Q_i(\varepsilon)),\:\tau^S(Q_i(1/3))]\cap Q_i(\varepsilon/2)=\emptyset
\end{array}\right\},\notag\\
	H_i^{(3)}(l)&=\left\{\# K_i\ge lM^2|x|^2,\: S(\rho_i)\in Q_i^R,\: S[\tau^S(B'_i),k]\in Q_i[0,5/9]\ \mbox{for all}\ k\in K_i\right\},	\notag\\
	H_i^{(4)}&=\left\{S[\rho_i,\tau^S(Q_{i+1})]\cap Q_{i}[0,11/18]=\emptyset\right\},\notag
\end{align}
where $\# A$ denotes the cardinality of set $A$. Moreover, $H_i(l)=H_i^{(1)}\cap H_i^{(2)}\cap H_i^{(3)}(l)\cap H_i^{(4)}$.
\end{dfn}

Note that, on the event $H_i(l)$, a local cut time $k\in K_i$ satisfies
\[S[\tau^S(Q_i),k]\cap S[k+1,\tau^S(Q_{i+1})]=\emptyset,\]
and thus $\mathrm{len}(\lambda_i)\ge \# K_i$ holds.

\begin{lem}\label{sw-lem-1}
Let $H_i(l)$ be as defined above. Then there exists constants $c>0$, $\varepsilon>0$, $l>0$ and $R'>0$ such that
\begin{equation}\label{sw-lem1-1}
\mathbf{P}^z\left(H_i(l)\mid F_i\right)\ge c,
\end{equation}
uniformly in $x$ and $M$ with $M|x|>R'$, $i\in\{1,2,\dots,N_M-1\}$ and $z\in\widetilde{Q}_i$.
\end{lem}

\begin{proof}
By the strong Markov property, we have that
\begin{align}
\lefteqn{\mathbf{P}^z\left(H_i(l)\mid F_i\right)}\nonumber\\
&\ge \inf_{z_1,z_2,z_3,z_4}\frac{1}{\mathbf{P}^z(F_i)}\prod_{j=1}^3\mathbf{P}^{z_j}\left(H_i^{(j)}\cap \{S[\tau^S(z_j),\tau^S(z_{j+1})]\subseteq B_i[-\varepsilon,1]\}\right)	\notag	\\
	&\qquad \times\mathbf{P}^{z_4}\left(H_i^{(4)}\cap\left\{\begin{array}{c}S[\tau^S(z_3),\tau^S(z_4)]\subseteq B_i[-\varepsilon,1],\: S(\tau^S(Q_{i+1}))\in \widetilde{Q}_{i+1},\\		
S[\tau^S(Q_{i+1}(-\varepsilon)),\:\tau^S(Q_{i+1})]\cap Q_{i+1}(-2\varepsilon)\end{array}\right\}\right)	\notag	\\
	&\ge \inf_{z_1}\mathbf{P}^{z_1}(H_i^{(1)})\inf_{z_2}\mathbf{P}^{z_2}(H_i^{(2)}\mid \{S[\tau^S(z_2),\tau^S(z_3)]\subseteq B_i[-\varepsilon,1]\})\inf_{z_3}\mathbf{P}^{z_3}(H_i^{(3)}(l))	\notag	\\
	&\qquad\times\inf_{z_4}\mathbf{P}^{z_4}\left(H_i^{(4)}\:\vline\:\begin{array}{c}
S[\tau^S(z_3),\tau^S(z_4)]\subseteq B_i[-\varepsilon,1],\:S(\tau^S(Q_{i+1}))\in \widetilde{Q}_{i+1},\\
S[\tau^S(Q_{i+1}(-\varepsilon)),\:\tau^S(Q_{i+1})]\cap Q_{i+1}(-2\varepsilon)\end{array}\right),\label{sw-lem1-1.2}
\end{align}
where the infima are taken over $z_1\in \widetilde{Q}_i$, $z_2\in \partial\widetilde{B}(z_1,2\varepsilon)$, $z_3\in Q_i^L$ and $z_4\in Q_i^R$ (see (\ref{sw-2}) for the definition of $\widetilde{B}(y,r)$).

First we estimate the conditional probability of $H_i^{(1)}$. Recall that $B(x,r)$ denotes the Euclidean ball of radius $r$ with center point $x$. We consider the event of $S$ up to the first exit time of the small box $\widetilde{B}(z_1,2\varepsilon)$. Let $k_1=\tau^S(B(z,\frac{\varepsilon}{2}M|x|)^c)$ and $k_2=\tau^S(\widetilde{B}(z_1,2\varepsilon)^c)$. Then
\begin{eqnarray}
\mathbf{P}^{z_1}(H_i^{(1)})	&\ge&\mathbf{P}^{z_1}\left(S\mbox{ has a nice cut time }k\in[k_1,k_2],\:0<k_2-\tau^S(Q_i)\le C'\varepsilon^2M^2|x|^2\right)\notag\\
&\ge&\mathbf{P}^{z_1}\left(\begin{array}{c}
S\mbox{ has a nice cut time }k\in[k_1,k_2],\\
\#S[\tau^S(Q_i),k_1]\ge C'^{-1}\varepsilon^2 M^2|x|^2,\:0<k_2-\tau^S(Q_i)\le C'\varepsilon^2 M^2|x|^2
\end{array}\right)\notag\\
&\ge& \mathbf{P}^{z_1}\left(S\mbox{ has a nice cut time }k\in[C'^{-1}\varepsilon^2M^2|x|^2,C'\varepsilon^2M^2|x|^2]\right)	\notag	\\
&& \qquad-\mathbf{P}^{z_1}\left(\#S[\tau^S(Q_i),k_1]\ge C'^{-1}\varepsilon^2 M^2|x|^2\right)\notag	\\
&& \qquad-\mathbf{P}^{z_1}\left(0<k_2-\tau^S(Q_i)\le C'\varepsilon^2 M^2|x|^2\right).	\label{sw-lem1-1.2+}
\end{eqnarray}
If we take $C'>1$ sufficiently large, then the second and third terms on the right-hand side of (\ref{sw-lem1-1.2+}) are bounded below by some small constant, while it follows from \cite[equation (1)]{Law-cut} that the first term is bounded below by some universal constant. Thus, we have
\begin{equation}\label{sw-lem1-1.3}
	\mathbf{P}^{z_1}(H_i^{(1)})\ge c_1
\end{equation}
for some constant $c_1>0$.

Second we consider the conditional probability of $H_i^{(2)}$. By (\ref{srwbound}), there exists some universal constant $C>0$ such that
\[\mathbf{P}^{z_2}\left(S\left[\tau^S(B(z,\varepsilon M|x|/2)^c),\tau^S(B'_i)\right]\cap (B(z,\varepsilon^2M|x|/2))\neq\emptyset\right)\le C\varepsilon^{d-2},\]
for $M|x|\ge \varepsilon^{-d}$. It follows from \eqref{gambler} that
\begin{align*}
c_2\varepsilon &\le \mathbf{P}^{z_2}\left(\tau^S(B'_i)<\tau^S(Q_{i+1}),\ S[\tau^S(z_2),\tau^S(z_3)]\subseteq B_i[-\varepsilon,1]\right)\notag\\
&\le \mathbf{P}^{z_2}\left(S[\tau^S(z_2),\tau^S(z_3)]\subseteq B_i[-\varepsilon,1]\right)\\
&\le c_3\varepsilon,
\end{align*}
uniformly in $z_2\in B(z_1,\varepsilon M|x|/2)$. Thus we have that
\begin{equation}\label{sw-lem1-1.4}
\mathbf{P}^{z_2}\left(H_i^{(2)}\:\vline\: S[\tau^S(z_2),\tau^S(z_3)]\subseteq B_i[-\varepsilon,1]\right)\ge \frac{c_2\varepsilon-C\varepsilon^{d-2}}{c_3\varepsilon}\ge c,
\end{equation}
for some constant $c>0$ and sufficiently small $\varepsilon$.

Again by \eqref{gambler}, it holds that
\begin{align}
\mathbf{P}^{z_4}&\left(H_i^{(4)}\:\vline\:\begin{array}{c}
S[\tau^S(z_4),\tau^S(Q_{i+1})]\subseteq B_i[-\varepsilon,1],\:S(\tau^S(Q_{i+1}))\in \widetilde{Q}_{i+1},\\
S[\tau^S(Q_{i+1}(-\varepsilon)),\tau^S(Q_{i+1})]\cap Q_{i+1}(-2\varepsilon)
\end{array}\right)\ge c	\label{sw-lem1-1.5}
\end{align}
for some constant $c>0$.

We will next derive a lower bound for $\mathbf{P}^{z_3}(H_i^{(3)})$ by applying the second moment method. We consider the ball $\mathcal{B}\coloneqq B\left(y,{M|x|}/{18}\right)$ and two independent simple random walks $R^1$ and $R^2$ with starting point $y$. Let $w_j=R^j(\tau^{R^j}(\mathcal{B}^c))$ for $j=1,2$. We define two events of $R^1$ and $R^2$ as follows:
\begin{align*}
I_1&=\left\{R^1[1,\tau^{R^1}(\mathcal{B}^c)]\cap R^2[1,\tau^{R^2}(\mathcal{B}^c)]=\emptyset\right\},\\
I_2&=\left\{\mathrm{dist}(\{w_1\},R^2[1,\tau^{R^2}(\mathcal{B}^c)])\vee \mathrm{dist}(\{w_2\},R^1[1,\tau^{R^1}(\mathcal{B}^c)])\ge M|x|/36\right\}.
\end{align*}
Let us denote by $P$ the joint distribution of $R^1$ and $R^2$. By \cite[Lemma 3.2]{BJ}, we have
\begin{equation}\label{sw-lem1-2.1}
\mathbf{P}(I_2\mid I_1)\ge c_4,
\end{equation}
while it follows from \cite[equation (3.2)]{Lawb} that $\mathbf{P}(I_1)\ge c_4$ for some constant $c_4>0$. On $I_2$, without loss of generality, we suppose that $|w_1-z_3|\le |w_2-z_3|$. Let
\begin{align*}
J_1=&\left\{\tau^{R^1}_{z_3}<\tau^{R^1}({B'_i}^c),\ \mathrm{dist}(R^1(k),l^1)\le M|x|/200 \mbox{ for all }k\in [\tau^{R^1}(\mathcal{B}^c),\tau^{R^1}_{z_3}]\right\}	\\
J_2=&\left\{R^2(\tau^{R_2}({B_i}^c))\in Q_i^R,\ \mathrm{dist}(R^2(k),l^2)\le M|x|/200 \mbox{ for all }k\in[\tau^{R^2}(\mathcal{B}^c),\tau^{R^2}({B'_i}^c)]\right\},
\end{align*}
where $l^1$ (respectively $l^2$) is the line segment between the points $w_1$ and $z_3$ (or between $w_2$ and $R^2(\tau^{R^2}({B'_i}^c))$, respectively). (See Figure \ref{sw-fig-3} for a depiction of $I_1\cap I_2\cap J_1\cap J_2$.)
\begin{figure}[t]
\centering
\includegraphics[width=0.35\textwidth]{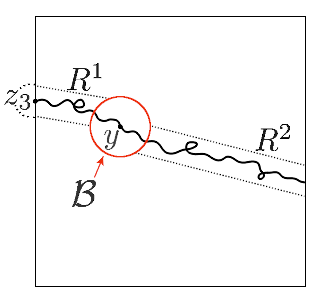}
\caption{Illustration of the event $I_1\cap I_2\cap J_1\cap J_2$.}\label{sw-fig-3}
\end{figure}
Since $\mathrm{dist}(w_2,{B'_i}^c)$ is comparable to $M|x|$, we have that
\[\mathbf{P}(J_2)\ge c',\]
for some constant $c'>0$. Moreover, by the strong Markov property and (\ref{srwbound2}),
\begin{align}
\mathbf{P}(J_1)&\ge \mathbf{P}\left(\mathrm{dist}(R^1(k),l^1)\le M|x|/200 \mbox{ for all }k\in [\tau^{R^1}(\mathcal{B}^c),\tau^{R^1}(B(z_3,M|x|/400))]\right)	\notag	\\
&\:\:\:\:\:\:\times \mathbf{P}^{R^1(\tau^{R^1}(B(z_3,M|x|/400)))}\left(\tau^{R^1}_{z_3}<\tau^{R^1}({B'_i}^c)\right)\times \mathbf{P}^{z_3}\left(R^1(\tau^{R^1}(B(z_3,M|x|/200)))\not\in B'_i\right)\notag	\\
&\ge c|w_1-z_3|^{2-d},\label{sw-lem1-3.2}
\end{align}
for some constant $c>0$. By the strong Markov property, we bound from below the expectation of $\#K_i$ on the event $A\coloneqq\{S(\rho_i)\in Q_i^R, S[\tau^S(B'_i),k]\in Q_i[0,5/9]\ \mbox{for all}\ k\in K_i\}$ by
\begin{align}
\mathbf{E}^{z_3}\left(\# K_i\mathbf{1}_A\right)&\ge \sum_{y\in B''_i} \mathbf{P}\left(I_1\cap I_2\cap J_1\cap J_2\right)\notag\\
&=\sum_{y\in B''_i} \mathbf{P}(I_1)\mathbf{P}(I_2\mid I_1)\mathbf{P}(J_1)\mathbf{P}(J_2)\notag\\
&\ge\sum_{y\in B''_i} c_4^2\cdot c|y-z_3|\cdot c'\notag\\
&\ge cM^2|x|^2.\hspace{240pt}\label{sw-lem1-4}
\end{align}
On the other hand, the first and second moment of $K_i$ is bounded above as follows. Since $|y-z_3|\ge \frac{1}{9}M|x|$ for $y\in B''_i$ and $z\in Q_i^L$,
\begin{align}
\mathbf{E}^{z_3}(\#K_i)&\le \mathbf{E}^{z_3}\left(\sum_{y\in B''_i}\mathbf{1}(\tau^S_y<\infty)\right)	\le \sum_{y\in B''_i}\mathbf{P}^{z_3}\left(\{\tau^S_y<\infty\}\right) \le \sum_{y\in B''_i}C|y-z_3|^{2-d}	\le CM^2 |x|^2,\label{sw-lem1-5}
\end{align}
\begin{align}
\mathbf{E}^{z_3}((\#K_i)^2)&\le \mathbf{E}^{z_3}\left(\left(\sum_{y\in B''_i}\mathbf{1}(\tau^S_y<\infty)\right)^2\right)\notag\\
		&\le CM^2|x|^2+\sum_{y\in B''_i}\sum_{y'\in B''_i}\left(\mathbf{P}^{z_3}(\tau^S_y<\tau^S_{y'}<\infty)+\mathbf{P}^{z_3}(\tau^S_{y'}<\tau^S_y<\infty)\right)\notag\\
		&\le CM^2|x|^2+\sum_{y\in B''_i}\sum_{y\in B''_i}\left(\mathbf{P}^{z_3}(\tau^S_y<\infty)+\mathbf{P}^{z_3}(\tau^S_{y'}<\infty)\right)\mathbf{P}^{y}(\tau^S_{y'}<\infty)	\notag	\\
		&\le CM^2|x|^2+\sum_{y\in B''_i}\sum_{k=1}^{\frac{1}{18}M|x|}\sum_{\substack{y\in B''_i\\|y-y'|=k}}\left(|y-z_3|^{2-d}+|y'-z_3|^{2-d}\right)k^{2-d}	\notag	\\
		&\qquad+\sum_{y\in B''_i}\sum_{k\ge\frac{1}{18}M|x|+1}\sum_{\substack{y\in B''_i\\|y-y'|=k}}\left(|y-z_3|^{2-d}+|y'-z_3|^{2-d}\right)(M|x|/9)^{2-d}	\notag	\\
		&\le C'M^4|x|^4,	\label{sw-lem1-6}
\end{align}
where $C$ and $C'$ depend only on $d$. Now, for $0\le \theta \le 1$, we have that
\begin{align*}
	\mathbf{E}^{z_3}&\left(\# K_i\mathbf{1}_A\right)\le \theta\mathbf{E}^{z_3}\left(\# K_i\mathbf{1}_A\right)+\mathbf{E}^{z_3}\left(\# K_i\mathbf{1}_A\mathbf{1}\left(\# K_i\mathbf{1}_A> \theta\mathbf{E}^{z_3}(\# K_i\mathbf{1}_A)\right)\right).
\end{align*}
From this, since $\# K_i\ge 0$, the Cauchy-Schwarz inequality yields
\begin{align*}
	\mathbf{P}^{z_3}(\{\#K_i> \theta\mathbf{E}^{z_3}(\# K_i\mathbf{1}_A)\}\cap A)&\ge \mathbf{P}^{z_3}(\{\#K_i\mathbf{1}_A> \theta\mathbf{E}^{z_3}(\# K_i\mathbf{1}_A)\}\cap A)	\\
	&\ge (1-\theta)^2 \frac{\mathbf{E}^{z_3}(\# K_i\mathbf{1}_A)^2}{\mathbf{E}^{z_3}((\# K_i)^2)}.
\end{align*}
By substituting (\ref{sw-lem1-4}), (\ref{sw-lem1-5}) and (\ref{sw-lem1-6}) into the above estimate, we obtain that
\begin{equation}\label{sw-lem1-7}
\mathbf{P}^{z_3}(H_i^{(3)}(l))\ge \mathbf{P}^{z_3}\left(\#K_i\ge \frac{l}{C}\mathbf{E}^{z_3}(\#K_i)\right)\ge\left(1-\frac{l}{C}\right)^2\frac{\mathbf{E}^{z_3}(\#K_i\mathbf{1}_A)^2}{\mathbf{E}^{z_3}((\#K_i)^2)}\ge c.
\end{equation}

Finally, substituting (\ref{sw-lem1-1.3}), (\ref{sw-lem1-1.4}), (\ref{sw-lem1-1.5}) and (\ref{sw-lem1-7}) into (\ref{sw-lem1-1.2}) gives (\ref{sw-lem1-1}).
\end{proof}

We are now ready to prove Proposition \ref{sw-prop}. Recall that $F_i$, $G_i$ and $H_i$ were defined in Definitions \ref{sw-def-1}, \ref{sw-def-3} and \ref{sw-def-5}, respectively. Let
\[U_{N_M}=\left\{S[\tau^S(Q_{2N_M+1}),\infty]\cap B(0,N_M\cdot M|x|)=\emptyset\right\},\]
and
\[\Theta=\Theta(C,l)=\left(\bigcap_{i=0}^{2N_M}F_i\right)\cap\left(\bigcap_{i=0}^{N_M}G_i(C)\right)\cap\left(\bigcap_{i=1}^{N_M-1}H_i(l)\right)\cap U_{N_M}.\]

\begin{proof}[Proof of Proposition \ref{sw-prop}]
We will first demonstrate that the bound $\tau_x\in[R^{-1}M|x|^2,RM|x|]$, as appears in the probability on the left-hand side of (\ref{sw-0}), holds on $\Theta(C,l)$. Suppose that $\Theta(C,l)$ occurs. By the definition of $F_i$, $i=0,1,\dots,2N_M$,
\[\mathrm{LE}(S[0,\tau^S_x])\cap S[\tau^S_x+1,\infty]=\emptyset\]
holds. Thus $x\in L$ and $\tau_x=\mathrm{len}(S[0,\tau^S_x])$. Let $k_i$ be a nice cut time of $S$ in $B_i$ (see Definition \ref{sw-def-4}), and recall that $\xi_i$ and $\xi'_i$ are as defined in (\ref{sw-def-xi}) and (\ref{sw-def-lambda}), respectively. Let
\begin{align*}
	s_i&=\inf\left\{n\ge 0\mathrel{:}\xi_{i-1}(n)\in S[\tau^S(Q_i),\tau^S(Q_{i+1})]\right\},\\
	t_i&=\sup\left\{n\in[\tau^S(Q_i),\tau^S(Q_{i+1})]\mathrel{:}S(n)=\xi_{i-1}(s_i)\right\}.
\end{align*}
Then we have that
\begin{equation}\label{sw-prop1-2}
\lambda_i=\mathrm{LE}\left(S[\tau^S(Q_{i}),k_i]\right)\oplus \mathrm{LE}\left(S[k_i,\tau^S(Q_{i+1})]\right),
\end{equation}
and also
\[\xi_i=\xi_{i-1}[0,s_i]\oplus \mathrm{LE}\left(S[t_i,k_i]\right)\oplus\mathrm{LE}\left(S[k_i,\tau^S(Q_{i+1})]\right)\subseteq \xi_{i-1}\cup S[t_i,k_i]\cup \lambda_{i},\]
for $i=1,2,\dots,N_M-1$, where we have applied (\ref{sw-prop1-2}) for the inclusion. Furthermore,
\[\xi_{N_M}=\xi_{N_M-1}[0,s_{N_M}]\oplus \mathrm{LE}(S[t_{N_M},\tau^S_x])\subseteq \xi_{N_M-1}\cup \lambda_{N_M}.\]
Thus, by induction, it follows that
\[	\bigcup_{i=1}^{N_M-1}\mathrm{LE}\left(S[k_i,\tau^S(Q_{i+1})]\right)\subseteq \xi_{N_M}\subseteq \xi_0\cup \bigcup_{i=1}^{N_M-1}\left(S[\tau^S(Q_i),\tau^S(Q_i(\varepsilon))]\cup \lambda_i\right)\cup \lambda_{N_M}.\]
Note that, on $H_i(l)$, $k'\in K_i$ is a cut time of the path $S[k_i,\tau^S(Q_{i+1})]$, and thus $S(k')\in LE(S[k_i,\tau^S(Q_{i+1})])$.
By the definition of $G_i(C)$ and $H_i(l)$ (see Definitions \ref{sw-def-3} and \ref{sw-def-5}, respectively), we have that
\begin{align*}
	\mathrm{len}(\xi_{N_M})&\ge \sum_{i=1}^{N_M-1} \#K_i \ge \frac12 lM|x|^2,\\
	\mathrm{len}(\xi_{N_M})&\le \mathrm{len}(\xi_0)+\sum_{i=1}^{N_M-1}\left(\#S[\tau^S(Q_i),\tau^S(Q_i(\varepsilon))]+\mathrm{len}(\lambda_i)\right)+\mathrm{len}(\lambda_{N_M})	\notag	\\
	&\le 3(C+\varepsilon^2)M|x|^2,
\end{align*}
on $\Theta(C,l)$. Hence choosing $R$ suitably large gives the desired bound on $\Theta(C,l)$.

Consequently, to complete the proof, it will be enough to show that $\mathbf{P}(\Theta)$ is bounded below by the right-hand side of (\ref{sw-0}). By Lemma \ref{sw-lem-1}, there exist constants $c_5,\varepsilon,l>0$ such that $\inf_{z\in \widetilde{Q}_i}\mathbf{P}^z(H_i(l)\mid F_i)\ge c_5$ for $i=1,2,\dots,N_M$. Moreover, by Lemma \ref{sw-lem-2}, there exists a constant $C>0$ such that $\inf_{z\in \widetilde{Q}_i}\mathbf{P}^z(G_i(C)\mid F_i)\ge 1-c_5/2$ for $i=0,1,\dots,N_M$. Thus we have
\begin{equation}\label{sw-prop1-5.1}
\inf_{z\in\widetilde{Q}_i}\mathbf{P}^z(G_i(C)\cap H_i(l)\mid F_i)\ge \frac{c_5}{2},\qquad i\in\{1,2,\dots,N_M-1\},
\end{equation}
\begin{equation}\label{sw-prop1-5.2}
\inf_{z\in\widetilde{Q}_i}\mathbf{P}^z(G_i(C)\mid F_i)\ge 1-\frac{c_5}{2},\qquad i=0,N_M.	
\end{equation}
As already noted in the proof of Lemma \ref{sw-lem-2}, we also have that
\begin{equation}\label{sw-prop1-6}
\inf_{z\in \widetilde{Q}_i}\mathbf{P}^z(F_i)\ge c_6\varepsilon,
\end{equation}
for all $i=1,2,\dots,N_M-1$, and a similar bound holds for $\mathbf{P}(F_0)$. And, repeating a similar argument to the lower bound for $\mathbf{P}^{z_3}(H_i^{(3)})$ from the proof of Lemma \ref{sw-lem-1}, from (\ref{sw-lem1-2.1}) to (\ref{sw-lem1-3.2}) we have that
\[	\inf_{z\in \widetilde{Q}_{N_M}}\mathbf{P}^z(F_{N_M})\ge c_6M^{2-d}|x|^{2-d},\]
where $c_6>0$ is adjusted if necessary. By combining these estimates on $\mathbf{P}(F_i)$ with (\ref{sw-prop1-5.1}) and (\ref{sw-prop1-5.2}), we obtain that
\[\inf_{z\in\widetilde{Q}_i}\mathbf{P}^z(F_i\cap G_i(C)\cap H_i(l))\ge \frac{c_5c_6\varepsilon}{2},\qquad i\in\{1,2,\dots,N_M-1\},\]
\[\mathbf{P}(F_0\cap G_0(C))\ge \left(1-\frac{c_5}{2}\right)c_6,\]
\[\inf_{z\in\widetilde{Q}_{N_M}}\mathbf{P}^z(F_{N_M}\cap G_{N_M}(C))\ge \left(1-\frac{c_5}{2}\right)c_6M^{2-d}|x|^{2-d}.\]
Furthermore, similarly to the case with $i=1,2,\dots,N_M-1$, it holds that
\[\inf_{z\in \widetilde{Q}_i}\mathbf{P}^z(F_i)\ge c_6\varepsilon,\qquad i\in\{N_M+1,N_M+2,\dots,2N_M\},\]
and it follows from (\ref{srwbound}) that
\[\inf_{z\in\widetilde{Q}_{2N_M+1}}\mathbf{P}^z(U_{N_M})\ge c\]
for some constant $c>0$. Finally, by the strong Markov property, we have that
\begin{align*}
\lefteqn{\mathbf{P}(\Theta(C,l))}\\
&\ge\mathbf{P}(F_0\cap G_0(C))\prod_{i=1}^{N_M-1}\inf_{z\in\widetilde{Q}_i}\mathbf{P}^{z}\left(F_i\cap G_i(C)\cap H_i(l)\right)\\
&\qquad\qquad\times\inf_{z\in\widetilde{Q}_{N_M}}\mathbf{P}^{z}(F_{N_M}\cap G_{N_M}(C))\times \prod_{i=N_M+1}^{2N_M}\inf_{z\in\widetilde{Q}_i}\mathbf{P}^{z}(F_i)\times \inf_{z\in\widetilde{Q}_{2N_M+1}}\mathbf{P}^{z}(U_{N_M})	\\
&\ge CM^{2-d}|x|^{2-d}e^{-\frac{c}{M}}\\
&\ge CM^{1-d/2}|x|^{2-d}e^{-\frac{c}{M}},
\end{align*}
where the third inequality holds simply because $M\leq 1$.
\end{proof}

\subsection{Lower bound for $M\geq 1$}

We now turn to the proof of the lower bound of Proposition \ref{Lprop} with $M\geq 1$. In particular, we will establish the following.

\begin{prop}\label{lower-M-large}
There exist constants $c,c',R\in(0,\infty)$ such that for every $x\in\mathbb{Z}^d\backslash\{0\}$ and $M\geq 1$,
\[\mathbf{P}\left(\tau_x\in[R^{-1}M|x|^2,RM|x|^2]\right)\ge c'(M|x|^2)^{1-d/2}\exp\left(-\frac{c}{M}\right).\]
\end{prop}
\bigskip

As in the previous section, the basic strategy involves the construction of a set of particular realisations of $L$ that we can show occurs with suitably high probability. To do this, we will use a certain reversibility property of simple random walk, as is set out in the next lemma. In the statement of this, for a finite path $\lambda = [ \lambda (0), \lambda (1), \cdots , \lambda (m) ]$, we write $\lambda^{R} = [ \lambda (m), \lambda (m-1), \cdots , \lambda (0) ]$ for its time reversal.

\begin{lem}\label{revlem} Let $x,z\in\mathbb{Z}^d$, $x\neq z$, and write $S^{x}$, $S^z$ for independent simple random walks in $\mathbb{Z}^{d}$ started at $x$, $z$, respectively. Moreover, write $\tau_x^z:=\inf\{j:\:S^z_j=x\}$, $\tau_z^x:=\inf\{j:\:S^x_j=z\}$,
\[ \sigma_{1} = \sup \{ j \le \tau^{x}_{z} \::\:S^{x}_{j} = x \},\quad u=\inf\{j\geq \tau_z^x\::\:S^x_j=x\},\quad\sigma_{2} = \sup \{ j<u\::\: S^{x}_{j} = z \}.\]
It then holds that
\begin{equation}\label{claim1}
\left\{ \lambda \::\: \mathbf{P} \left( \left( S^{z} [0, \tau^{z}_{x} ] \right)^{R} = \lambda \:\vline\:\tau^{z}_{x} < \infty \right) > 0 \right\} = \left\{ \lambda \::\: \mathbf{P} \left( S^{x} [\sigma_{1}, \sigma_{2} ]  = \lambda \:\vline\:\tau^{z}_{x} < \infty \right) > 0 \right\},
\end{equation}
and, denoting the set above $\Lambda$,
\begin{equation}\label{claim2}
\mathbf{P} \left( \left( S^{z} [0, \tau^{z}_{x} ] \right)^{R} = \lambda \:\vline\: \tau^{z}_{x} < \infty \right) = \mathbf{P} \left( S^{x} [\sigma_{1}, \sigma_{2} ] = \lambda \:\vline\: \tau^{x}_{z} < \infty \right),\qquad \forall \lambda \in \Lambda.
\end{equation}
\end{lem}
\begin{proof}
Since \eqref{claim1} is easy to see, we only check \eqref{claim2}. Take $\lambda = [ \lambda (0), \lambda (1), \cdots , \lambda (m) ] \in \Lambda$. Note that $\mathbf{P} ( \tau^{z}_{x} < \infty ) = \mathbf{P} ( \tau^{x}_{z} < \infty )$ by symmetry. It follows that
\[\mathbf{P} \left( \left( S^{z} [0, \tau^{z}_{x} ] \right)^{R} = \lambda, \  \tau^{z}_{x} < \infty \right) = \mathbf{P} \left( S^{z} [0, \tau^{z}_{x} ]  = \lambda^{R}, \  \tau^{z}_{x} < \infty \right) = \mathbf{P} \left( S^{z} [0, \tau^{z}_{x} ]  = \lambda^{R} \right)  = (2d)^{-m}.\]
On the other hand, we have
\begin{align*}
\mathbf{P} \left( S^{x} [\sigma_{1}, \sigma_{2} ] = \lambda, \  \tau^{x}_{z} < \infty \right) &= \mathbf{P} \left( S^{x} [\sigma_{1}, \sigma_{2} ] = \lambda \right) \\
&= \sum_{ k \ge 0}  \mathbf{P} \left( S^{x} [\sigma_{1}, \sigma_{2} ] = \lambda, \ \sigma_{1} = k  \right) \\
&  = \sum_{ k \ge 0}  \mathbf{P} \left( z \notin S^{x} [0, k], \ S^{x}_{k} = x, \ S^{x} [k, k + m] = \lambda, \ \sigma_{2} = k+m \right)  \\
& =  \sum_{ k \ge 0}  \mathbf{P} \left( z \notin S^{x} [0, k], \ S^{x}_{k} = x, \ S^{x} [k, k + m] = \lambda, \
S^{x}_{k +m} = z, \ F \right),
\end{align*}
where $F:= \{ z \notin S^{x} [k+m +1 ,  u')\}$ with $u' = \inf \{ j \ge k+m :\:S^{x}_{j} = x \}$. Therefore, the Markov property ensures that
\begin{align*}
&\mathbf{P} \left( z \notin S^{x} [0, k], \ S^{x}_{k} = x, \ S^{x} [k, k + m] = \lambda, \
S^{x}_{k +m} = z, \ F \right) \\
&= \mathbf{P} \left( z \notin S^{x} [0, k], \ S^{x}_{k} = x \right) \mathbf{P} \left( S^{x} [0,  m] = \lambda \right) \mathbf{P} (F') \\
&= (2d)^{-m} \mathbf{P} \left( z \notin S^{x} [0, k], \ S^{x}_{k} = x \right) \mathbf{P} (F' ),
\end{align*}
where $F':=\{z\not\in S^{z} [1 ,  \tau^z_x]\}$. Writing
\[ \xi_{x} = \inf \{ j \ge 1 :\: S^{x}_{j} = x \} \qquad \text{ and } \qquad p = \mathbf{P} \left( \xi_{x} < \infty, \ z \notin S^{x} [0, \xi_{x}] \right),\]
we note that
$$\sum_{k \ge 0 } \mathbf{P} \left( z \notin S^{x} [0, k], \ S^{x}_{k} = x \right) = \frac{1}{1-p}.$$
Moreover, by symmetry again, it holds that $\mathbf{P} (F')=1-p$. Hence we conclude that
\[\mathbf{P} \left( S^{x} [\sigma_{1}, \sigma_{2} ] = \lambda, \  \tau^{x}_{z} < \infty \right)  = (2d )^{-m},\]
which gives \eqref{claim2}.
\end{proof}

In order to explain our application of the previous result, we need to introduce some notation. First, let $x\in\mathbb{Z}^d\backslash\{0\}$ and $M\geq 1$. Moreover, set $J=C\sqrt{M|x|^2}$ for some $C\ge1$ that will be determined later, and, for $i\in\mathbb{Z}$, write $\widehat{b}_i=(2iJ,0,\dots,0)\in\mathbb{R}^d$ and
\[\widehat{B}_i=B_\infty\left(\widehat{b}_i,J\right),\]
which represent adjacent cubes of side length $2J$. We also introduce the following smaller cubes centred at $b'=(\frac{9}{4}J,\frac{J}{2},0,\dots,0)\in\mathbb{R}^d$,
\[\widehat{B}'=B_\infty(b',{J}/{6}),\qquad \widehat{B}''=B_\infty(b',J/18)\]
Note that $\widehat{B}''\subset\widehat{B}'\subset \widehat{B}_1$. See Figure \ref{cubesfig} for a sketch showing the cubes $\widehat{B}_{-1}$, $\widehat{B}_0$, $\widehat{B}_1$ and $\widehat{B}'$, as well as some of the other objects that we now define. In particular, we introduce a collection of surfaces:
\[Q^*=\left\{\frac{5}{2}J\right\}\times[-J,J]^{d-1},\]
\[Q_*=\{3J\}\times[-J,J]^{d-1},\]
\[Q=\{3J\}\times \left[-\frac{J}{4},\frac{J}{4}\right]\times[-J,J]^{d-2}\subset Q_*,\]
\[\widetilde{Q}=\left\{J\right\}\times[-J,J]^{d-1},\qquad Q'=\left\{(y^1-J/16,y^2,\cdots,y^d)\mathrel{:}(y^1,y^2,\cdots,y^d)\in\widetilde{Q}\right\},\]
\[\widetilde{Q}_{\pm}=\left\{J\right\}\times\left[\pm\frac{J}{2}-\frac{J}{8},\pm\frac{J}{2}+\frac{J}{8}\right]\times[-J,J]^{d-2}\subset\widetilde{Q},\]
\[Q_{-1}=\widehat{B}_{-1}\cap \{-3J\}\times \mathbb{R}^{d-1};\]
the hyperplane $\mathbb{H}^{(1)}_{85J/36}$, where for $a\in\mathbb{R}$ and $i\in\{1,\dots,d\}$, we denote
\[\mathbb{H}^{(i)}_a=\left\{(x_1,\cdots,x_d)\in\mathbb{R}^d\mathrel{:}x_i=a\right\}\]
(see Figure \ref{figure-B'} below for the location of $\mathbb{H}^{(1)}_{85J/36}$ in particular); and also the following regions:
\[D_\pm=\left[-J,\frac{49}{16}J\right]\times[-J,J]^{d-1}\setminus \left[\frac{15}{16}J,\frac{5}{2}J\right]\times\left[-\frac{J}{2}\mp\frac{J}{2},\frac{J}{2}\mp\frac{J}{2}\right]\times[-J,J]^{d-2},\]
\[\widetilde{D}_\pm=D_\pm\cap \left[\frac{15}{16}J,\infty\right)\times\mathbb{R}^{d-1}.\]
We highlight that $D_+$ is shown as the shaded region on Figure \ref{cubesfig}.

\begin{figure}[t]
\centering
\includegraphics[width=0.75\linewidth]{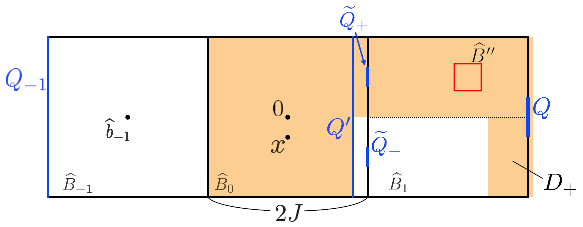}
\vspace{0pt}
\caption{Cubes and other regions appearing in the proof of Proposition \ref{lower-M-large}.}\label{cubesfig}
\end{figure}

Roughly speaking, to establish the main result of this section, we will show that, with high enough probability, the loop-erased random walk $L$ passes from 0 to (somewhere close to) $Q$ through $D_+$, spending a suitable time in $\widehat{B}''$ on the way, before returning to $x$ through $D_-$, and then escapes to $\infty$ via $Q_{-1}$. To make this precise, we will consider an event based on the simple random walk started from $0$; see Figure \ref{roughfig}. Controlling the probability of this will involve an appeal to Lemma \ref{revlem}, through which we obtain a bound that depends on three independent random walks, one started from $0$ and two started from $x$ (see Lemma \ref{3path} below).

\begin{figure}[t]
		\begin{center}
		\includegraphics[width=0.8\linewidth]{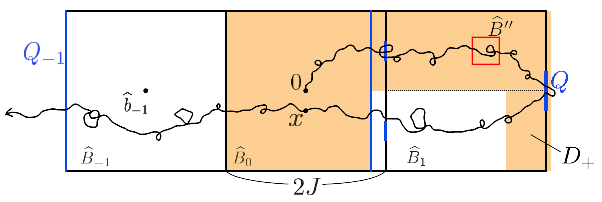}\end{center}
\vspace{-15pt}
		\caption{A sketch of a realisation of $S^0$ yielding $\tau_x\ge M|x|^2$.}\label{roughfig}
\end{figure}

Concerning notation, as in the statement of Lemma \ref{revlem}, for each $z\in\mathbb{Z}^d$, we will write $S^z$ for a simple random walk started from $z$. We assume that the elements of the collection $(S^z)_{z\in\mathbb{Z}^d}$ are independent. We moreover write $(\widetilde{S}^z)_{z\in\mz}$ for an independent copy of  $(S^z)_{z\in\mathbb{Z}^d}$. We also set
\[\tau^z_A:=\inf\{k\ge 0\mathrel{:}S^z_k\in A\},\qquad \sigma^z_A=\sup\{k\ge 0\mathrel{:}S^z_k\in A\},\]
$\tau^z_x=\tau^z_{\{x\}}$ and $\sigma^z_x=\sigma^z_{\{x\}}$. A particularly important point in the argument that follows is given by
\[\rho=S^0_{\tau^0_Q},\]
i.e.\ the location where $S^0$ hits $Q$, which is defined when $\tau_Q^0<\infty$. Additionally, we introduce
\[\widetilde{\tau}=\inf\left\{k\ge 0\mathrel{:}\widetilde{S}^x_k\in \mathbb{R}^d\setminus(\widehat{B}_0\cup\widehat{B}_{-1})\right\},\]
and, to describe a collection of local cut points for a path $\lambda$,
\[\Gamma(\lambda[i,j])=\left\{\lambda(k)\mathrel{:}\lambda[i,k]\cap\lambda[k+1,j]=\emptyset\right\}.\]

The following result gives the key decomposition of the simple random walk underlying $L$ that we will consider later in the section. It already takes into account the time-reversal property of Lemma \ref{revlem}. We will break the complicated event that appears in the statement into several more convenient pieces below.

\begin{lem}\label{3path} In the setting described above, $\mathbf{P}(\tau_x\ge M|x|^2)$ is bounded below by the probability of the following event:
\[\left\{\begin{array}{c}
         \tau_Q^0<\infty,\:S^0[0,\tau^0_Q]\subset D_+,\:S^0[0,\sigma^0_{\widehat{B}''}]\cap \mathbb{H}^{(1)}_{85J/36}=\emptyset,\:S^0[\tau^0_{Q^*},\tau^0_Q]\cap \mathbb{H}^{(1)}_{85J/36}=\emptyset,\\\#(\Gamma(S^0[0,\tau^0_Q])\cap \widehat{B}'')\ge M|x|^2,\:\tau^x_\rho<\infty,\:S^x[0,\sigma^x_\rho]\subset D_{-},\\
        (S^0[0,\tau^0_Q]\cap S^x[0,\sigma^x_\rho])\cap\widehat{B}_0=\emptyset,\:\widetilde{S}^x\cap (S^0[0,\tau^0_Q]\cup S^x[0,\sigma^x_\rho])=\emptyset\end{array}\\\right\}.\]
\end{lem}

\begin{proof} Clearly,
\[\left\{\begin{array}{c}
         \tau_Q^0<\tau^0_x<\infty,\:S^0[0,\tau^0_Q]\subset D_+,\:S^0[0,\sigma^0_{\widehat{B}''}]\cap \mathbb{H}^{(1)}_{85J/36}=\emptyset,\:S^0[\tau^0_{Q^*},\tau^0_Q]\cap \mathbb{H}^{(1)}_{85J/36}=\emptyset,\\
         \#(\Gamma(S^0[0,\tau^0_Q])\cap \widehat{B}'')\ge M|x|^2,\: S^0[\tau^0_Q,\tau^0_x]\subset D_{-},\\
(S^0[0,\tau^0_Q]\cap S^0[\tau^0_Q,\tau^0_x])\cap\widehat{B}_0=\emptyset,\:S^0[\tau^0_x,\infty)\cap (S^0[0,\tau^0_Q]\cup S^0[\tau^0_Q,\tau^0_x])=\emptyset
       \end{array}\\\right\}\]
is a subset of the event $\{\tau_x\ge M|x|^2\}$. Now, conditioning on the value of $\rho$ and applying the strong Markov property at times $\tau_Q^0$ and $\tau^0_x$, we have that the probability of the above event is equal to
\[\sum_{z\in Q}\mathbf{P}\left(\begin{array}{c}
         \tau_Q^0<\infty,\:\rho=z,\:\tau^z_x<\infty,\\
         S^0[0,\tau^0_Q]\subset D_+,\:S^0[0,\sigma^0_{\widehat{B}''}]\cap \mathbb{H}^{(1)}_{85J/36}=\emptyset,\:S^0[\tau^0_{Q^*},\tau^0_Q]\cap \mathbb{H}^{(1)}_{85J/36}=\emptyset,\\
         \#(\Gamma(S^0[0,\tau^0_Q])\cap \widehat{B}'')\ge M|x|^2,  \:S^z[0,\tau^z_x]\subset D_{-},\\
         (S^0[0,\tau^0_Q]\cap S^z[0,\tau^z_x])\cap\widehat{B}_0=\emptyset,\:\widetilde{S}^x\cap (S^0[0,\tau^0_Q]\cup S^z[0,\tau^z_x])=\emptyset\end{array}\\\right).\]
Applying Lemma \ref{revlem}, we can replace $\tau^z_x$ and $S^z[0,\tau^z_x]$ in the above expression by $\tau^x_z$ and $S^x[\sigma_1,\sigma_2]$, respectively, where $\sigma_1$, $\sigma_2$ are defined as in the statement of that result. Since $0\leq \sigma_1\leq \sigma_2\leq \sigma^x_z$, it holds that $S^x[\sigma_1,\sigma_2]\subseteq S^x[0,\sigma_z^x]$. Consequently, we obtain that the above sum is bounded below by
\[\sum_{z\in Q}\mathbf{P}\left(\begin{array}{c}
         \tau_Q^0<\infty,\:\rho=z,\:\tau^x_z<\infty,\\
         S^0[0,\tau^0_Q]\subset D_+,\:S^0[0,\sigma^0_{\widehat{B}''}]\cap \mathbb{H}^{(1)}_{85J/36}=\emptyset,\:S^0[\tau^0_{Q^*},\tau^0_Q]\cap \mathbb{H}^{(1)}_{85J/36}=\emptyset,\\
         \#(\Gamma(S^0[0,\tau^0_Q])\cap \widehat{B}'')\ge M|x|^2,\:S^x[0,\sigma_z^x]\subset D_{-},\\
         (S^0[0,\tau^0_Q]\cap S^x[0,\sigma_z^x])\cap\widehat{B}_0=\emptyset,\:\widetilde{S}^x\cap (S^0[0,\tau^0_Q]\cup S^x[0,\sigma_z^x])=\emptyset\end{array}\\\right),\]
and replacing the sum with a union inside the probability completes the proof.
\end{proof}

Now, we will rewrite the event we defined in the statement of Lemma \ref{3path} as the intersection of various smaller events concerning $S^0$, $S^x$ and $\widetilde{S}^x$. For convenience, we will write
\[\eta_0\coloneqq S^0_{\tau^0_{\widetilde{Q}}},\qquad \eta_x\coloneqq S^x_{\tau^x_{\widetilde{Q}}},\qquad\widetilde{\eta}\coloneqq \widetilde{S}^x_{\widetilde{\tau}}\]
in the remainder of this subsection. We moreover define the event $E_1$ by setting
\[E_1=\left\{\begin{array}{c}
    \tau_{\widetilde{Q}}^0<\infty,\:\eta_0\in\widetilde{Q}_+,\:\tau^x_\rho<\infty,\:\eta_x\in\widetilde{Q}_-,\:\widetilde{\tau}<\infty,\:\widetilde{\eta}\in Q_{-1},\\
	(S^0[0,\tau^0_{\widetilde{Q}}]\cap S^x[0,\tau^x_{\widetilde{Q}}])\cap\widehat{B}_0=\emptyset,	\\
	S^0[0,\tau^0_{\widetilde{Q}}]\subset D_+,\:S^x[0,\tau^x_{\widetilde{Q}}]\subset D_{-},\:\widetilde{S}^x[0,\widetilde{\tau}]\cap (S^0[0,\tau^0_{\widetilde{Q}}]\cup S^x[0,\tau^x_{\widetilde{Q}}])=\emptyset\end{array}\right\}.\]
On $E_1$, the paths $S^0$, $S^x$ and $\widetilde{S}^x$ do not have an intersection and move along the different courses until they first exit the union of $\widehat{B}_{-1}$ and $\widehat{B}_0$.

Next, we will define some events that restrict the behavior of $S^0$ after $\tau^0_{\widetilde{Q}}$. First, recall that $b'=(\frac{9}{4}J,\frac{J}{2},0,\cdots,0)\in\mathbb{R}^d$ and $\widehat{B}'=B_\infty(b',\frac{J}{6})$. We define the ``left'' and ``right'' faces of $\widetilde{B}'$ by
\[Q_L=\left\{\frac{25}{12}J\right\}\times\left[\frac{J}{3},\frac{2}{3}J\right]\times\left[-\frac{J}{6},\frac{J}{6}\right]^{d-2},\qquad
	Q_R=\left\{\frac{29}{12}J\right\}\times\left[\frac{J}{3},\frac{2}{3}J\right]\times\left[-\frac{J}{6},\frac{J}{6}\right]^{d-2}.\]
Moreover, we define a subset of $\widetilde{D}_+$ by setting
\[\widehat{B}'_L=\left[\frac{17}{18}J,\frac{79}{36}J\right]\times\left[\frac{J}{3},\frac{2}{3}J\right]\times\left[-\frac{J}{2},\frac{J}{2}\right]^{d-2}.\]
See Figure \ref{figure-B'}. Then, writing $u^y=\inf\{n\ge\tau^y_{\widetilde{Q}}\mathrel{:}S^y_n\in Q'\}$ and
	$\sigma'=\inf\{n\ge\tau^{y}_{\widehat{B}''}\mathrel{:}S^{y}_n\in(\widehat{B}')^c\}$, let
\begin{align}
	F_1(y)=&\left\{\tau^{y}_{Q_L}<u^y,S^{y}[\tau^y_{\widetilde{Q}},\tau^{y}_{Q_L}]\subset\widehat{B}'_L\right\},\label{draft-SW-1}\\
	F_2(y)=&\left\{\begin{array}{c}
\tau^{y}_{\widehat{B}''}<\inf\{n\ge\tau^{y}_{Q_L}\mathrel{:}S^{y}_n\in(\widehat{B}'_L)^c\}<\infty,\:\sigma'\in Q_R,\\
		\#\left\{k\in[\tau^{y}_{\widehat{B}''},\sigma']\mathrel{:}S^{y}[\tau^{y}_{Q_L},k]\cap S^{y}[k+1,\sigma']=\emptyset,\:S^y_k\in\widehat{B}''\right\}\ge M|x|^2\end{array}\right\},\label{draft-SW-2}\\
	F_3(y)=&\left\{\tau^y_{Q^*}<\tau^{y}_Q<\infty,\:S^{y}[\sigma',\tau^{y}_Q]\subset D_+\cap\left[\frac{85}{36}J,\infty\right)\times\mathbb{R}^{d-1}\right\},	\label{draft-SW-3}
\end{align}
and set $E_2=F_1(0)\cap F_2(0)\cap F_3(0)$. In particular, on the event $E_2$, we have the existence of cut points of $S[\tau^y_{Q_L},\sigma']$ contained in $\widehat{B}''$. Finally let
\begin{align*}
	E_3=&\left\{\tau^x_\rho<\infty,\:
	S^x[\tau^x_{\widetilde{Q}},\sigma^x_\rho]\subset \widetilde{D}_{-}\right\},\\
	E_4=&\left\{\widetilde{S}^x[\widetilde{\tau},\infty]\cap (S^0[0,\tau^0_Q]\cup S^x[0,\sigma^x_\rho])=\emptyset\right\},
\end{align*}
be events that restrict the regions where $S^x$ and $\widetilde{S}^x$ can explore, respectively.

\begin{figure}[t]
		\centering
		\includegraphics[width=0.4\linewidth]{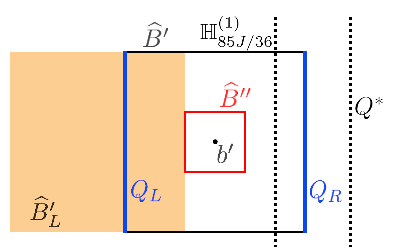}
				\vspace{0pt}
		\caption{Illustration of the sets used in controlling the number of cut points of $S^0$ in $\widehat{B}''$.}\label{figure-B'}
\end{figure}

We continue by checking the local cut points that we construct on the event $F_2(0)$ are cut points of the loop-erasure of $S^0[0,\tau^0_Q]$. Note that on the event $E_2$, it follows from the definition of $F_2(0)$ and $F_3(0)$ that
\[S[\tau^0_{Q^*},\tau^0_Q]\cap \hpl{1}{85J/36}=\emptyset,\qquad S[\sigma',\tau^0_{Q_*}]\cap \widehat{B}''=\emptyset.\]
Moreover, on the event $E_1\cap E_2$, we have that
\begin{itemize}
	\item $S^0[0,\tau^0_{\widetilde{Q}}]\cap Q^*=\emptyset$,
	\item $S^0[\tau^0_{\widetilde{Q}},\tau^0_{Q_L}]\cap Q^*=\emptyset$,
	\item $S^0[\tau^0_{Q_L},\sigma']\cap Q^*=\emptyset$.
\end{itemize}
The first and second statements follow from the definitions of $E_1$ and $F_1(0)$, respectively, while the third statement is derived from the definitions of $F_2(0)$ and $\sigma'$ (recall the definitions of the sets defined above, which are also shown in Figure \ref{figure-B'}). From these statements, we immediately conclude that
\[S^0[0,\sigma']\cap Q^*=\emptyset. \]
For the rest of the path $S^0[0,\tau^0_Q]$, the definition of $F_3(0)$ implies that
\[S^0[\sigma',\tau^0_{Q^*}]\cap \widehat{B}''=\emptyset,\qquad S^0[\tau^0_{Q^*},\tau^0_Q]\cap \hpl{1}{85/36J}=\emptyset.\]
Combining the preceding three statements, we see that, on $E_1\cap E_2$,
\[S^0[0,\sigma^0_{\widehat{B}''}]\cap\hpl{1}{85/36J}=\emptyset,\qquad S^0[\tau^0_{Q^*},\tau^0_Q]\cap \hpl{1}{85J/36}=\emptyset,\]
where we recall that is $\sigma^0_{\widehat{B}''}$ be the last exit time of $\widehat{B}''$ by $S^0$ (we assume here that $S^0$ is stopped at $\tau^0_Q$). Thus, the local cut points of the event $F_2(0)$ are indeed cut points of the loop-erasure of $S^0[0,\tau^0_Q]$ and the probability of the event we defined in the statement of Lemma \ref{3path} is bounded below by
\[\mathbf{P}(E_1\cap E_2\cap E_3\cap E_4).\]
In what follows, we will bound below this probability below. To start with, we will prove that $S^0$, $S^x$ and $\widetilde{S}^x$ do not have an intersection and are separated in a cube with positive probability. Let $T^z_r=\tau^z_{B_\infty(0,r)}$. We define the event $G_n$ by setting
\[G_n=\left\{S^0[0,T^0_n]\cap S^x[0,T^x_n]=S^0[0,T^0_n]\cap \widetilde{S}^x[0,\widetilde{T}^x_n]=S^x[0,T^x_n]\cap \widetilde{S}^x[1,\widetilde{T}^x_n]=\emptyset\right\},\]
and let $Z_n$ be given by
\[\min\left\{d(S^0_{T^0_n},S^x[0,T^x_n]\cup \widetilde{S}^x[0,\widetilde{T}^x_n]), d(S^x_{T^x_n},S^0[0,T^0_n]\cup \widetilde{S}^x[0,\widetilde{T}^x_n]), d(\widetilde{S}^x_{\widetilde{T}^x_n},S^0[0,T^0_n]\cup S^x[0,T^x_n])\right\},\]
where $d$ here is the Euclidean distance, i.e.\ $Z_n$ is the minimum of the distance between the point from which either $S^0$, $S^x$ or $\widetilde{S}^x$ exits $B_\infty(0,n)$ and the union of the other two paths up to their exit times.

\begin{lem}\label{SW-separate}
There exists $c>0$ and $n_0\ge 1$ such that: for all $n\ge n_0$,
\[\mathbf{P}\left(G_n\cap\left\{Z_n\ge\frac{n}{2}\right\}\right)\ge c.\]
\end{lem}

\begin{proof}
For readability, we assume that $x=(0,|x|,0,\cdots,0)$. (Other cases will follow by a small modification of the argument.) We follow the idea of \cite[Lemma 3.2]{BJ}. Let $e_1=(1,0,0,\cdots,0)\in\mz$ and $e_2=(0,1,0,\cdots,0)\in\mz$. We define the event $I_1$ by setting
\[I_1=\left\{S^0_i=ie_2,\:S_i^x=ie_1,\:\widetilde{S}_i^x=-ie_1\ \mbox{for}\ 1\le i\le k\right\},\]
where $k\ge 1$ will be fixed later. Then we have $\mathbf{P}(I_1)=(2d)^{-3k}$.

We will show that the probability that $S^0$, $S^x$ and $\widetilde{S}^x$ do not intersect before they first exit from $B_\infty(0,n)$ conditioned on $I_1$ is bounded above by arbitrarily small $\varepsilon$ by taking $k$ sufficiently large. Let $K(j)$ be the number of intersections of $S^0[j,T^0_n]$, $S^x[j,T^x_n]$ and $\widetilde{S}^x[j,\widetilde{T}^x_n]$, i.e.\
\[K(j)=\#\left((S^0[j,T^0_n]\cap S^x[j,T^x_n])\cup(S^0[j,T^0_n]\cap \widetilde{S}^x[j,\widetilde{T}^x_n])\cup(S^x[j,T^x_n]\cap \widetilde{S}^x[\max\{1,j\},\widetilde{T}^x_n])\right).\]
Then by the Markov inequality,
\begin{align}
	\mathbf{P}(K(0)>0\mid I_1)&\le \mathbf{P}(K(k)\ge 1\mid I_1)\notag\\
        &\le \mathbf{E}(K(k)\mid I_1)	\notag	\\
		&\le 3\max_{x,y:\:d(x,y)\geq k}\mathbf{E}\left(\#(S^{x}[0,T^{x}_n]\cap S^{y}[0,T^{y}_n])\right)	\notag	\\
		&\le 3\max_{x,y:\:d(x,y)\geq k}\sum_{m=0}^\infty\sum_{m'=0}^\infty\sum_{z\in\mz}\mathbf{P}^{x}(S^{x}_m=z)\mathbf{P}^y(S^y_{m'}=z)	\notag	\\
		&\le 3\max_{x,y:\:d(x,y)\geq k}\sum_{m=0}^\infty\sum_{m'=0}^\infty \mathbf{P}^{x}(S^{x}_{m+m'}=y)	\notag	\\
		&\le 3\sum_{l=1}^\infty l\cdot Cl^{-d/2}e^{-ck^2/l}\notag\\
        &\le C\sum_{l=1}^\infty l^{1-d/2}e^{-ck^2/l},\label{lemma-SW-1}
\end{align}
for some $C$, $c>0$, where we applied the Gaussian estimate for the off-diagonal heat kernel of the simple random walk on $\mz$ for the last inequality (see \eqref{srwbound0}). Since $d\ge 5$, the right-hand side of (\ref{lemma-SW-1}) converges to 0 as $k\to \infty$.

Our next step is to construct subsets where each simple random walk path is constrained to move until it first exits from $B_\infty(0,n)$. We define by
\[H_L=\left\{-\frac{n}{2}\right\}\times\left[-\frac{n}{4},\frac{n}{4}\right]^{d-1},\qquad	H_R=\left\{\frac{n}{2}\right\}\times\left[-\frac{n}{4},\frac{n}{4}\right]^{d-1},\]
the subsets of the left and right face of $B_\infty(0,\frac{n}{2})$ in the direction of $x_1$-axis, respectively, and by
\[H_+=\left[-\frac{n}{4},\frac{n}{4}\right]\times\left\{\frac{n}{2}\right\}\times\left[-\frac{n}{4},\frac{n}{4}\right]^{d-2},\]
the subset of the upper face of $B_\infty(0,\frac{n}{2})$ in the direction of $x_2$-axis. Let
\begin{align*}
	I^0_2&=\left\{S^0_{T^0_{n/2}}\in H_+,\:S^0_{T^0_n}\in\hpl{2}{n},\:S^0[T^0_{n/2},T^0_n]\cap (\hpl{1}{n/3}\cup \hpl{1}{-n/3}\cup \hpl{2}{n/3})=\emptyset\right\},	\\
	I^x_2&=\left\{S^x_{T^x_{n/2}}\in H_R,\:S^x_{T^x_n}\in\hpl{1}{n},\:S^x[T^x_{n/2},T^x_n]\cap (\hpl{1}{n/3}\cup\hpl{2}{n/3})=\emptyset\right\},	\\
	\widetilde{I}^x_2&=\left\{\widetilde{S}^x_{\widetilde{T}^x_{n/2}}\in H_L,\:\widetilde{S}^x_{\widetilde{T}^x_n}\in\hpl{1}{-n},\:\widetilde{S}^x[\widetilde{T}^x_{n/2},\:\widetilde{T}^x_n]\cap(\hpl{1}{-n/3}\cup\hpl{2}{n/3})=\emptyset\right\},
\end{align*}
and $I_2=I_2^0\cap I_2^x\cap \widetilde{I}^x_2$. It is an elementary exercise to check that there exists some $\varepsilon>0$ such that $\mathbf{P}(I_2\mid I_1)\ge \varepsilon$ holds uniformly in $n\ge 1$ and $x\in\mz$. Note that on the event $I_1\cap I_2$, it holds that
\[S^x_{T^x_n}\in\hpl{1}{n},\qquad S^0[0,T^0_n]\cup \widetilde{S}^x[0,\widetilde{T}^x_n]\subset (-\infty,\frac{n}{2}]\times\mathbb{R}^{d-1},\]
so that $d(S^x_{T^x_n},S^0[0,T^0_n]\cup \widetilde{S}^x[0,\widetilde{T}^x_n])\ge \frac{n}{2}$. Similarly, the same bound holds for $d(S^0_{T^0_n},S^x[0,T^x_n]\cup \widetilde{S}^x[0,\widetilde{T}^x_n])$ and $d(\widetilde{S}^x_{\widetilde{T}^x_n},S^0[0,T^0_n]\cup S^x[0,T^x_n])$. Thus we have $Z_n\ge\frac{n}{2}$ on the event in question. Finally, by taking $k$ large so that $\mathbf{P}(K(0)>0\mid I_1)\le \frac{\varepsilon}{2}$, we obtain that
\begin{align*}
	\mathbf{P}\left(G_n\cap\left\{Z_n\ge \frac{n}{2}\right\}\right)&\ge \mathbf{P}(I_1\cap I_2\cap \{K(0)=0\})	\\
	&=\mathbf{P}(I_1)\left(\mathbf{P}(I_2\mid I_1)-\mathbf{P}(K(0)>0\mid I_1)\right)\\
	&\ge (2d)^{-3k}\varepsilon/2,
\end{align*}
which completes the proof.
\end{proof}

We are now ready to complete the proof of the main result of this subsection.

\begin{proof}[Proof of Proposition \ref{lower-M-large}]
For any $R>1$, we have that
\[	\mathbf{P}\left(\tau_x\in[R^{-1}M|x|^2,RM|x|^2]\right)\geq \mathbf{P}\left(\tau_x\geq M|x|^2\right)-\mathbf{P}\left(\tau_x>R M|x|^2\right)\]
Now, by the argument used to prove Proposition \ref{upper-M-large}, we have that
\[\mathbf{P}\left(\tau_x>R M|x|^2\right)\leq C\left(R M|x|^2\right)^{1-d/2}.\]
Thus, by taking $R$ suitably large, and applying Lemma \ref{3path} and the argument above Lemma \ref{SW-separate}, to complete the proof it suffices to prove that
\[\mathbf{P}\left(E_1\cap E_2\cap E_3\cap E_4\right)\ge cJ^{2-d}\]
for some constant $c$.
	
First, we derive an estimate for $\mathbf{P}(E_1)$ from the result of Lemma \ref{SW-separate}. We take $C\ge 1$ large so that $J=C\sqrt{M|x|^2}\ge 3n_0$. We then define the event $F'$ by setting
\[F'=\left\{\begin{array}{c}
             \tau_{\widetilde{Q}}^0<\infty,\:\eta_0\in\widetilde{Q}_+,\:\tau^x_\rho<\infty,\:\eta_x\in\widetilde{Q}_-,\:\widetilde{\tau}<\infty,\:\widetilde{\eta}\in Q_{-1},  \\
             S^0[T^0_{J/3},\tau^0_{\widetilde{Q}}]\cap\left(\hpl{1}{J/9}\cup\hpl{2}{5J/36}\right)=\emptyset,\:
		S^x[T^x_{J/3},\tau^x_{\widetilde{Q}}]\cap\left(\hpl{1}{J/9}\cup\hpl{2}{5J/36}\right)=\emptyset, \\
               \widetilde{S}^x[\widetilde{T}^x_{J/3},\widetilde{\tau}]\cap\hpl{1}{-J/9}=\emptyset
            \end{array}\right\}.\]
Recall the definition of the events $I_1$ and $I_2$ and the random variable $K(j)$ from the proof of Lemma \ref{SW-separate}. It is straightforward to check that if we take $n=J/3$, then
\[I_1\cap I_2\cap \{K(0)=0\}\cap F'\subset E_1.\]
Moreover, by the strong Markov property and the approximation to Brownian motion, there exists some constant $c>0$ such that $\mathbf{P}(F'\mid I_1\cap I_2\cap \{K(0)=0\})\ge c$, uniformly in $x$ and $M$. By Lemma \ref{SW-separate}, we thus obtain that
\begin{equation}
\mathbf{P}(E_1)\ge \mathbf{P}\left(F'\mid I_1\cap I_2\cap \{K(0)=0\}\right)\mathbf{P}(I_1\cap I_2\cap \{K(0)=0\})\ge c^2.\label{proof-SW-0}
\end{equation}

Next we estimate $\mathbf{P}(E_2\mid E_1)$. It follows from the strong Markov property that
\[\mathbf{P}\left(E_2\mid E_1\right)\ge \inf_{a_1\in \widetilde{Q}_+}\mathbf{P}^{a_1}\left(F_1(a_1)\cap F_2(a_2)\cap F_3(a_3)\right),\]
where $F_1(y)$, $F_2(y)$ and $F_3(y)$ are as defined in (\ref{draft-SW-1}), (\ref{draft-SW-2}) and (\ref{draft-SW-3}), respectively. Thus it suffices to bound from below the right-hand side of the above inequality. We begin with a lower bound for $\mathbf{P}^{a_1}(F_1)$. By the gambler's ruin estimate (\ref{gambler}), we have that $\mathbf{P}^{a_1}(F_1)\ge c$ for some universal constant $c>0$. Next, applying a similar argument to that used to obtain (\ref{sw-lem1-7}) in the proof of Lemma \ref{sw-lem-1} and the strong Markov property, we obtain that $\mathbf{P}^{a_1}(F_2\mid F_1)\ge c$. Finally, again by (\ref{gambler}) and the strong Markov property, we have that $\mathbf{P}^{a_1}(F_3\mid F_1\cap F_2)\ge c$.
Since $c>0$ does not depend on $a_1\in\widetilde{Q}_+$, we can conclude that
\begin{equation}\label{proof-SW-2.2}
	\mathbf{P}^{a_1}(E_2\mid E_1)\ge c^3.
\end{equation}

By the independence of $S^0$ and $S^x$ and the strong Markov property, we also have that
\[\mathbf{P}(E_3\mid E_1\cap E_2)\ge\inf_{a_2\in\widetilde{Q}_-,\:z\in Q}\mathbf{P}^{a_2}\left(\tau^{a_2}_z<\infty,\:S^{a_2}[0,\sigma_z^{a_2}]\subset\widetilde{D}_-\right).\]
Let $a_2\in\widetilde{Q}_-$ and $z\in Q$. Again by the strong Markov property,
\begin{equation}
\mathbf{P}^{a_2}\left(\tau^{a_2}_z<\infty,\:S^{a_2}[0,\sigma_z^{a_2}]\subset\widetilde{D}_-\right)=\mathbf{P}^{a_2}\left(\tau^{a_2}_z<\infty,\:S^{a_2}[0,\tau_z^{a_2}]\subset\widetilde{D}_-\right)\mathbf{P}^z\left(S^z[0,\sigma^z_z]\subset\widetilde{D}_-\right).\label{proof-SW-3}
\end{equation}
We will give lower bounds for the two probabilities in the right-hand side. Let $l_{a_2,z}$ be the piecewise linear curve that runs from $a_2$ in the direction of $e_2$ until its second coordinate reaches $5J/2$, and then runs along the line from that point to $z$. Similarly to (\ref{sw-lem1-3.2}) in Lemma \ref{sw-lem-1}, we obtain that
\begin{align}
	\mathbf{P}^{a_2}\left(\tau^{a_2}_z<\infty,S^{a_2}[0,\tau_z^{a_2}]\subset\widetilde{D}_-\right)
		&\ge \mathbf{P}^{a_2}\left(\tau^{a_2}_z<\infty,\mathrm{dist}(S^{a_2}(k),l_{a_2,z})\le J/16\ \mbox{for all}\ k\in[0,\tau^{a_2}_z]\right)	\notag	\\
		&\ge cJ^{2-d},\label{proof-SW-4}
\end{align}
uniformly in $a_2$ and $z$ for some $c>0$. Furthermore, we have that
\begin{align*}
\mathbf{P}^z\left(S^z[0,\sigma^z_z]\subset\widetilde{D}_-\right)&\ge 1-\mathbf{P}^z\left(S^z[0,\sigma^z_z]\cap B(z,J/16)^c\neq\emptyset\right)	\\
		&\ge 1-\sup_{w\in\partial B(z,J/16)}\mathbf{P}^w(\tau^w_z<\infty)	\\
		&\ge 1-\frac{a}{G(0)}J^{2-d},
\end{align*}
where we applied (\ref{srwbound2}) with $n\to\infty$ to the last inequality. Thus, by increasing the value of the constant $C>0$ in $J=C\sqrt{M|x|^2}$ if necessary, we have that
\begin{equation}
\mathbf{P}^z\left(S^z[0,\sigma^z_z]\subset\widetilde{D}_-\right)\ge c,	\label{proof-SW-5}
\end{equation}
for some uniform constant $c>0$. Plugging (\ref{proof-SW-4}) and (\ref{proof-SW-5}) into (\ref{proof-SW-3}) yields that
\begin{equation}
\mathbf{P}(E_3\mid E_1\cap E_2)\ge cJ^{2-d}.	\label{proof-SW-6}
\end{equation}

Now we will give a lower bound for $\mathbf{P}(E_4\mid E_1\cap E_2\cap E_3)$. Recall that $Q_{-1}=\widehat{B}_{-1}\cap \{-3J\}\times \mathbb{R}^{d-1}$. By the strong Markov property and the definition of the events $E_1$, $E_2$ and $E_3$, we have that
\[\mathbf{P}(E_4\mid E_1\cap E_2\cap E_3)\ge \inf_{a_3\in Q_{-1}}\widetilde{\mathbf{P}}^{a_3}(\widetilde{S}^{a_3}[0,\infty]\cap (D_+\cup D_-)=\emptyset).\]
From this, it is an easy application of (\ref{srwbound}) to deduce that
\begin{equation}
\mathbf{P}\left(E_4\mid E_1\cap E_2\cap E_3\right)\ge c,	\label{proof-SW-7}
\end{equation}
for some universal constant $c>0$.

Finally by multiplying (\ref{proof-SW-0}), (\ref{proof-SW-2.2}), (\ref{proof-SW-6}) and (\ref{proof-SW-7}), we obtain the desired lower bound.
\end{proof}

\section{Heat kernel estimates for the associated random walk}\label{sec4}

The aim of this section is to prove Theorem \ref{mainres}. As explained in the introduction, the main input concerning the loop-erased random walk will be Theorem \ref{thm:main1}. To estimate $\mathbb{P}(X_t^\mathcal{G}=x)$ using the decomposition at \eqref{decomp}, we also require control over $P^\mathcal{G}(X_t^\mathcal{G}=L_m)$, where in this section $(L_m)_{m\geq 0}$ is always the infinite LERW started from 0. Since the structure of the graph $\mathcal{G}$ is simply that of $\mathbb{Z}_+$ equipped with nearest-neighbour bonds, we have the obvious identity
\[P^\mathcal{G}(X_t^\mathcal{G}=L_m)=q_t(0,m),\]
where $(q_t(x,y))_{x,y\in\mathbb{Z}_+,\:t>0}$ gives the transition probabilities of the continuous-time simple random walk on $\mathbb{Z}_+$ with unit mean holding times. For this, we have the following estimates from \cite{Barbook}. (We note that although the result we will cite in \cite{Barbook} is stated for the simple random walk on $\mathbb{Z}$, it is easy to adapt to apply to the half-space $\mathbb{Z}_+$.)

\begin{lem}\label{qtlem}
For any $\varepsilon > 0$, there exist constants $c_1,c_2,c_3,c_4,c_5,c_6\in(0,\infty)$ such that for every $m\in\mathbb{Z}_+$ and $t\geq \varepsilon m$,
\[q_t(0,m)\leq c_1\left(1\wedge t^{-1/2}\right)\exp\left(-\frac{c_2m^2}{1\vee t}\right)\]
and also
\[q_t(0,m)\geq c_3\left(1\wedge t^{-1/2}\right)\exp\left(-\frac{c_4m^2}{1\vee t}\right).\]
Moreover, for $m\geq 1$ and $t<\varepsilon m$, we have that
\[q_t(0,m)\leq c_5\exp\left(-c_6m\left(1+\log(m/t)\right)\right).\]
\end{lem}
\begin{proof}
From \cite[Theorem 6.28(b)]{Barbook}, we obtain the relevant bounds for $t\geq 1\vee m$. Moreover, the bounds for $m=0$, $t\in (0,1)$, follow from \cite[Theorem 6.28(d)]{Barbook}. As for $m\geq 1$, $t\in (\varepsilon m,m)$, we can apply \cite[Theorem 6.28(c)]{Barbook} to deduce that $q_t(0,m)$ is bounded above and below by an expression of the form:
\[c\exp\left(-c^{-1}m\left(1+\log(m/t)\right)\right).\]
This can be bounded above and below by an expression of the form $c\exp(-c^{-1}m)$, and that in turn by $c(1\wedge t^{-1/2})\exp(-\frac{c^{-1}m^2}{1\vee t})$, uniformly over the range of $m$ and $t$ considered. This completes the proof of the first two inequalities in the statement of the lemma. The third inequality is given by again applying \cite[Theorem 6.28(c)]{Barbook}.
\end{proof}

We are now ready to proceed with the proof of Theorem \ref{mainres}.

\begin{proof}[Proof of Theorem \ref{mainres}]
Clearly, if $x=0$, then Lemma \ref{qtlem} immediately yields
\[\mathbb{P}\left(X_t^\mathcal{G}=x\right)=q_t(0,0)\asymp 1\wedge t^{-1/2},\]
which gives the result in this case.

We next suppose $x\neq 0$. In this case, applying Lemma \ref{qtlem} with $\varepsilon =1$, we find that
\begin{eqnarray}
  \mathbb{P}\left(X_t^\mathcal{G}=x\right) &=& \sum_{m=1}^\infty \mathbb{P}\left(X_t^\mathcal{G}=L_m\right)\mathbf{P}\left(L_m=x\right)\nonumber\\
  &=&\sum_{m=1}^\infty q_t(0,m)\mathbf{P}\left(L_m=x\right)\nonumber\\
   &\leq &c_1\left(1\wedge t^{-1/2}\right)\sum_{m=1}^t\exp\left(-\frac{c_2m^2}{1\vee t}\right)\mathbf{P}\left(L_m=x\right)\nonumber\\
   &&\qquad +c_3\sum_{m={t+1}}^\infty \exp\left(-c_4m\left(1+\log(m/t)\right)\right)\mathbf{P}\left(L_m=x\right).\label{twosums}
   \end{eqnarray}
Now, the second sum here is readily bounded as follows:
\[c_3\sum_{m={t+1}}^\infty \exp\left(-c_4m\left(1+\log(m/t)\right)\right)\mathbf{P}\left(L_m=x\right)\leq c_3\sum_{m={t+1}}^\infty \exp\left(-c_4m\right)\leq c_3\exp\left(-c_5t\right).\]
Moreover, since we are assuming $t\geq \varepsilon|x|\geq \varepsilon$, the final expression is readily bounded above by one of the form
\[c_6\left(1\wedge|x|^{2-d}\right)\left(1\wedge t^{-1/2}\right)\exp\left(-c_7\left(\frac{|x|^4}{1\vee t}\right)^{1/3}\right).\]
Thus, to complete the proof of upper bound in the statement of Theorem \ref{mainres}, it remains to derive a similar bound for the first sum on the right-hand side at \eqref{twosums}. For this, we have that
   \begin{eqnarray}
     \lefteqn{c_1\left(1\wedge t^{-1/2}\right)\sum_{m=1}^t\exp\left(-\frac{c_2m^2}{1\vee t}\right)\mathbf{P}\left(L_m=x\right)}\nonumber\\
     &\leq &c_1\left(1\wedge t^{-1/2}\right)\sum_{k=0}^\infty\exp\left(-\frac{c_2(2^k)^2}{1\vee t}\right)\sum_{m=2^k}^{2^{k+1}-1}
\mathbf{P}\left(L_m=x\right)\nonumber\\
 &\leq &c_1\left(1\wedge t^{-1/2}\right)\sum_{k=0}^\infty\exp\left(-\frac{c_2(2^k)^2}{1\vee t}\right)
(2^k)^{1-d/2}\exp\left(-\frac{c_3|x|^2}{2^k}\right)\nonumber\\
 &\leq &c_1\left(1\wedge t^{-1/2}\right)\sum_{m=1}^\infty m^{-d/2}\exp\left(-\frac{c_2m^2}{1\vee t}-\frac{c_3|x|^2}{m}\right)\nonumber\\
 &\leq &c_1\left(1\wedge t^{-1/2}\right)\int_{1}^\infty u^{-d/2}\exp\left(-\frac{c_2u^2}{1\vee t}-\frac{c_3|x|^2}{u}\right)du,\label{b1}
\end{eqnarray}
where we have applied Theorem \ref{thm:main1} for the second inequality. To bound the integral, we first note that, for any $\delta>0$, it is possible to find a constant $C<\infty$ such that $a^{d/2}\leq Ce^{\delta a}$ for all $a\geq 0$. In particular, choosing $\delta=c_3/2$, this implies that
\[u^{-d/2}=|x|^{-d}\left(\frac{u}{|x|^2}\right)^{-d/2}\leq C|x|^{-d}\exp\left(\frac{c_3|x|^2}{2u}\right).\]
Hence, applying this estimate and the change of variable $v=u/((1\vee t)|x|^2)^{1/3}$, we obtain
\begin{eqnarray}
\lefteqn{\int_{1}^\infty u^{-d/2}\exp\left(-\frac{c_2u^2}{1\vee t}-\frac{c_3|x|^2}{u}\right)du}\nonumber\\
&\leq&C|x|^{-d}\int_{1}^\infty \exp\left(-\frac{c_2u^2}{1\vee t}-\frac{c_3|x|^2}{2u}\right)du\nonumber\\
&\leq &C|x|^{2-d} \left(\frac{|x|^4}{1\vee t}\right)^{-1/3}\int_{0}^\infty \exp\left(-\left(c_2v^2+\frac{c_3}{2v}\right)\times \left(\frac{|x|^4}{1\vee t}\right)^{1/3}\right)dv.\label{b2}
\end{eqnarray}
Now, let $f(v):=c_2v^2+\frac{c_3}{2v}$, and note that this is a function that has a unique minimum $v_0$ on $(0,\infty)$ such that $f(v_0)>0$. Thus, for $|x|^4\geq1\vee t$, the remaining integral above is estimated as follows:
\begin{eqnarray*}
\lefteqn{\int_{0}^\infty \exp\left(-\left(c_2v^2+\frac{c_3}{2v}\right)\times \left(\frac{|x|^4}{1\vee t}\right)^{1/3}\right)dv}\\
&\leq& \int_{0}^\infty \exp\left(-\left(f(v)-f(v_0)\right)\right)dv\exp\left(-f(v_0)\left(\frac{|x|^4}{1\vee t}\right)^{1/3}\right)\\
&=&C\exp\left(-c\left(\frac{|x|^4}{1\vee t}\right)^{1/3}\right).
\end{eqnarray*}
Putting this together with \eqref{b1} and \eqref{b2}, we deduce the desired result in the range $|x|^4\geq 1\vee t$. If $|x|^4< 1\vee t$, then we follow a simpler argument to deduce:
\begin{eqnarray}
  \mathbb{P}\left(X_t^\mathcal{G}=x\right) &=& \sum_{m=1}^\infty q_t(0,m)\mathbf{P}\left(L_m=x\right)\nonumber\\
  &\leq &c_1\left(1\wedge t^{-1/2}\right) \sum_{m=1}^\infty \mathbf{P}\left(L_m=x\right)\nonumber\\
  &=&c_1\left(1\wedge t^{-1/2}\right)\mathbf{P}\left(L_m=x\mbox{ for some }m\geq 0\right)\nonumber\\
  &\leq&c_1\left(1\wedge t^{-1/2}\right)\mathbf{P}\left(S_m=x\mbox{ for some }m\geq 0\right)\nonumber\\
  &\leq&c_1\left(1\wedge t^{-1/2}\right)|x|^{2-d},\nonumber
\end{eqnarray}
where we have applied Lemma \ref{qtlem} for the first inequality, and \eqref{srwbound2} for the third. This is enough to establish that the upper bound of Theorem \ref{mainres} holds in this case as well.

For the lower bound when $x\neq 0$, we follow a similar argument to the upper bound, but with additional care about the range of summation/integration. In what follows, we set $\alpha=c_4/c_3$, where, here and for the rest of the proof, $c_3$, $c_4$ are the constants of Theorem \ref{thm:main1}. Clearly, we can assume that $c_3\leq 1<c_4$, so that $\alpha>1$. Recall that we are also assuming $t\geq \varepsilon|x|$, and without loss of generality, we may suppose $\varepsilon\in (0,1)$. Applying the bounds of Lemma \ref{qtlem} with $\varepsilon$ given by
\[\varepsilon':=\min\left\{\frac{\varepsilon}{1+\alpha^2},\frac{\alpha\varepsilon^{4/3}}{4c_4(1+\alpha^2)}\right\},\]
we deduce that
\begin{eqnarray*}
\mathbb{P}\left(X_t^\mathcal{G}=x\right) &=& \sum_{m=1}^\infty q_t(0,m)\mathbf{P}\left(L_m=x\right)\\
&\geq&c\left(1\wedge t^{-1/2}\right)\sum_{m=1}^{\lfloor t/\varepsilon'\rfloor}\exp\left(-\frac{Cm^2}{1\vee t}\right) \mathbf{P}\left(L_m=x\right)\\
&\geq&c\left(1\wedge t^{-1/2}\right)\sum_{k=0}^{\lfloor \log_\alpha(\lfloor t/\varepsilon'\rfloor )\rfloor-1}\exp\left(-\frac{C(\alpha^k)^2}{1\vee t}\right)\sum_{m=\lceil \alpha^k\rceil}^{\lfloor \alpha^{k+1}\rfloor}\mathbf{P}\left(L_m=x\right),
\end{eqnarray*}
where for the second inequality, we have applied that
\[\left[1,\lfloor t/\varepsilon'\rfloor\right]\supseteq \bigcup_{k=0}^{\lfloor \log_\alpha(\lfloor t/\varepsilon'\rfloor )\rfloor-1}\left[\alpha^k,\alpha^{k+1}\right]\]
and the observation that each $m$ can appear in at most two of the intervals $[\lceil \alpha^k\rceil,\lfloor \alpha^{k+1}\rfloor]$. (We also note that our choice of $c$ ensures $\lfloor \log_\alpha(\lfloor t/\varepsilon'\rfloor )\rfloor-1\geq 1$, and so the sum is non-empty.) Consequently, applying  Theorem \ref{thm:main1} with $n=\alpha^k/c_3$, we find that
\begin{eqnarray*}
\mathbb{P}\left(X_t^\mathcal{G}=x\right)&\geq &c\left(1\wedge t^{-1/2}\right)\sum_{k=0\vee \lceil\log_\alpha (c_3|x|)\rceil}^{\lfloor \log_\alpha(\lfloor t/\varepsilon'\rfloor )\rfloor-1}\exp\left(-\frac{C(\alpha^k)^2}{1\vee t}\right)(\alpha^k)^{1-d/2}\exp\left(-\frac{C|x|^2}{\alpha^k}\right)\\
&\geq &c\left(1\wedge t^{-1/2}\right)\sum_{m=1\vee \lceil \alpha c_3 |x|\rceil}^{\lfloor\alpha^{-1}\lfloor t/\varepsilon'\rfloor\rfloor } m^{-d/2}\exp\left(-\frac{Cm^2}{1\vee t}-\frac{C|x|^2}{m}\right)\nonumber\\
 &\geq &c\left(1\wedge t^{-1/2}\right)\int_{2c_4|x|}^{\alpha t/(1+\alpha^2)\varepsilon'} u^{-d/2}\exp\left(-\frac{Cu^2}{1\vee t}-\frac{C|x|^2}{u}\right)du,\nonumber
\end{eqnarray*}
where we have used that $ 1\vee \lceil \alpha c_3 |x|\rceil= 1\vee \lceil c_4 |x|\rceil=\lceil c_4 |x|\rceil$ to obtain the bottom limit of the integral, and the choice of $\varepsilon'$ to obtain the top one. Making the change of variable $v=u\varepsilon^{4/3}/((1\vee t)|x|^2)^{1/3}$ yields a lower bound for the integral of
\[c\left(((1\vee t)|x|^2)^{1/3}\right)^{1-d/2}\int_{2c_4\varepsilon \times\left(\varepsilon|x|/t\right)^{1/3}}^{\frac{\alpha\varepsilon^{7/3}}{(1+\alpha^2)\varepsilon'}\times\left(t/\varepsilon|x|\right)^{2/3}}
v^{-d/2}\exp\left(-C\left(v^2+\frac{1}{v}\right)\times\left(\frac{|x|^4}{1\vee t}\right)^{1/3}\right)dv,\]
and, since $t\geq \varepsilon |x|$, our choice of $\varepsilon'$ implies that this is bounded below by
\begin{eqnarray*}
\lefteqn{c\left(((1\vee t)|x|^2)^{1/3}\right)^{1-d/2}\int_{2c_4\varepsilon}^{4c_4\varepsilon}
v^{-d/2}\exp\left(-C\left(v^2+\frac{1}{v}\right)\times\left(\frac{|x|^4}{1\vee t}\right)^{1/3}\right)dv}\\
&\geq &c((1\vee t)|x|^2)^{1/3-d/6}\exp\left(-C\left(\frac{|x|^4}{1\vee t}\right)^{1/3}\right).\hspace{140pt}
\end{eqnarray*}
Hence, if $|x|^4\geq 1\vee t$, we can put the pieces together to find that
\begin{eqnarray*}
\mathbb{P}\left(X_t^\mathcal{G}=x\right)&\geq &c\left(1\wedge t^{-1/2}\right)((1\vee t)|x|^2)^{1/3-d/6}\left(\frac{|x|^4}{1\vee t}\right)^{1/3-d/6}\exp\left(-C\left(\frac{|x|^4}{1\vee t}\right)^{1/3}\right)\\
&=&c\left(1\wedge t^{-1/2}\right)|x|^{2-d}\exp\left(-C\left(\frac{|x|^4}{1\vee t}\right)^{1/3}\right),
\end{eqnarray*}
as required. Finally, for $|x|^4<1\vee t$, continuing to suppose that $c_4$ is the constant of Theorem \ref{thm:main1}, we have that
\begin{eqnarray*}
\mathbb{P}\left(X_t^\mathcal{G}=x\right)&\geq &\sum_{m=1}^{\lfloor\sqrt{4c_4^2t}\rfloor}q_t(0,m)\mathbf{P}\left(L_m=x\right)\\
&\geq &c\left(1\wedge t^{-1/2}\right) \sum_{m=1}^{\lfloor\sqrt{4c_4^2t}\rfloor} \mathbf{P}\left(L_m=x\right)\\
  &  \geq&c\left(1\wedge t^{-1/2}\right)\sum_{m=\lceil c_3\lceil|x|\rceil\rceil}^{\lfloor c_4\lceil|x|\rceil\rfloor} \mathbf{P}\left(L_m=x\right)\\
  &\geq&c\left(1\wedge t^{-1/2}\right)|x|^{2-d},
\end{eqnarray*}
where we have applied Lemma \ref{qtlem} with $\varepsilon=1/4c_4^2$ for the second inequality, that $c_4\lceil|x|\rceil\leq 2c_4|x|\leq \sqrt{4c_4^2t}$ for the third, and Theorem \ref{thm:main1} for the final one.
\end{proof}

\subsection*{Acknowledgements} DC was supported by JSPS Grant-in-Aid for Scientific Research (C) 19K03540 and the Research Institute for Mathematical Sciences, an International Joint Usage/Research Center located in Kyoto University. DS was supported by JSPS Grant-in-Aid for Scientific Research (C) 22K03336, JSPS Grant-in-Aid for Scientific Research (B) 22H01128 and 21H00989. SW is supported by JST, the Establishment of University Fellowships Towards the Creation of Science Technology Innovation, Grant Number JPMJFS2123.

\bibliographystyle{amsplain}
\bibliography{biblio}

\end{document}